\documentclass[11pt, psamsfonts]{amsart}
\usepackage{amsmath}
\usepackage{amsthm}
\usepackage{amssymb}
\usepackage{amscd}
\usepackage{amsfonts}
\usepackage{amsbsy}
\usepackage{epsfig}

\newtheorem{theorem}{Theorem}
\newtheorem{remark}{Remark}
\newtheorem{proposition}{Proposition}
\newtheorem{lemma}{Lemma}
\newtheorem{corollary}{Corollary}
\newtheorem{claim}{Claim}
\newtheorem*{claim*}{Claim}

\newfont\bbf{msbm10 at 12pt}
\def\c{{\bf C}}
\def\eps{\varepsilon}
\def\phi{\varphi}
\def\R{{\mathbb R}}

\def\N{{\mathbb N}}
\def\Z{{\mathbb Z}}

\def\E{{\mathcal E}}

\def\L{{\mathcal L}}
\def\H{{\mathcal H}}

\def\P{{\mathcal P}}
\def\Q{{\mathcal Q}}

\def\D{{\mathcal D}}
\def\M{{\mathcal M}}

\def\S{{\mathcal S}}

\def\es{{\emptyset}}
\def\sm{\setminus}

\def\lev{\mbox{level}}

\def\sCrit{\mbox{\tiny Crit}}

\def\dm{\mathfrak D}

\def\supp{\mbox{\rm supp}}

\def\orb{\mbox{\rm orb}}
\def\Crit{\mbox{\rm Crit}}

\def\conv{\mbox{conv\ }}
\def\bd{\partial }
\def\le{\leqslant}
\def\ge{\geqslant}
\def\htop{h_{top}}
\def\laps{\mbox{laps}}

\newcommand{\st}{such that }
\newcommand{\hold}{H\"{o}lder }

\begin{document}
\bibliographystyle{plain}
\title[Equilibrium states for interval maps:
the potential $-t \log |Df|$]{Equilibrium states for interval maps:\\ the potential $\mathbf{-t\, { \rm\bf{log} } |Df|}$}
\author{Henk Bruin, Mike Todd}
\thanks{
This research was supported by EPSRC grant GR/S91147/01}

\begin{abstract}
\noindent Let $f:I \to I$ be a $C^2$ multimodal interval map satisfying polynomial growth of the derivatives along critical
orbits. We prove the existence and uniqueness of
equilibrium states for the potential $\phi_t:x\mapsto
-t\log|Df(x)|$  for $t$ close to $1$, and also that the pressure function
$t \mapsto P(\phi_t)$ is analytic on an appropriate interval near $t = 1$.
\end{abstract}

\maketitle

\section{Introduction}\label{sec:intro}
Thermodynamic formalism ties potential functions $\phi$ to
invariant measures of a dynamical system $(X,f)$. The aim is to
identify and prove uniqueness of a measure $\mu_\phi$ that
maximises the {\em free energy}, i.e., the sum of the entropy and
the integral over the potential.  In other words
\[
h_{\mu_\phi}(f) + \int_X \phi~d\mu_\phi  = P(\phi) := \sup_{\nu
\in \M_{erg}} \left\{  h_{\nu}(f) + \int_X \phi \
d\nu:-\int_X\phi~d\nu<\infty \right\}
\]
where $\M_{erg}$ is the set of all ergodic $f$-invariant Borel
probability measures. Such measures are called {\em equilibrium
states}, and $P(\phi)$ is the {\em pressure}. This theory
was developed by Sinai, Ruelle and Bowen \cite{Sinai, Bowen, Ruelle} in
the context of H\"older potentials on hyperbolic dynamical
systems, and has been applied to Axiom A systems, Anosov
diffeomorphisms and other systems too, see e.g.\
\cite{Baladi,Kbook} for more recent expositions. Apart from
uniqueness, it was shown in this context that the density
$\frac{d\mu_\phi}{dm_\phi}$ of the invariant measure with respect
to $\phi$-conformal measure $m_\phi$ is a fixed point of the
transfer operator $(\L_\phi h)(x) = \sum_{f(y) = x} e^{\phi(y)}\
h(y)$. Moreover, $\mu_\phi$ is a Gibbs measure, i.e., there is a
constant $K > 0$ \st
\[
\frac 1K \le \frac{\mu_\phi(\c_n)}{e^{\phi_n(x) - n P(\phi)} } \le
K
\]
for all $n \in \N$, all $n$-cylinder sets $\c_n$ and any $x \in
\c_n$. Here $\phi_n(x):=\phi(f^{n-1}(x))+ \cdots +\phi(x)$.

In this paper we are interested in interval maps $(I,f)$ with
nonempty set $\Crit$ of critical points. These maps are, at best,
only non-uniformly hyperbolic.  We say that $c$ is a non-flat
critical point of $f$ if there exists a diffeomorphism $g_c:\R \to
\R$ with $g_c(0)=0$ and $1<\ell_c<\infty$ \st for $x$ close to
$c$, $f(x)=f(c)\pm|\phi_c(x-c)|^{\ell_c}$.  The value of
$\ell_c$ is known as the \emph{critical order} of $c$.
Let $\ell_{max} = \max\{ \ell_c : c \in \Crit\}$.
Throughout, $\H$ will be the collection of $C^2$ interval maps
with finitely many branches and only non-flat critical points.
There is a finite partition $\P_1$ into maximal intervals on which
$f$ is monotone.  Let us call this partition the {\em
branch partition}. We will assume throughout that $\vee_n \P_n$
generates the Borel $\sigma$-algebra. Note that if $f\in \H$ is
$C^2$ and has no attracting cycles then $\vee_n \P_n$ generates
the Borel $\sigma$-algebra, see \cite{MSbook}.  (The $C^2$
assumption precludes wandering sets, which are not very
interesting from the measure theoretic point of view anyway.)

The principal examples of maps in $\H$ are unimodal maps with
non-flat critical point.  Equilibrium states (in particular of the
potential $\phi_t := -t\log |Df|$) have been studied in this case by
various authors  \cite{HKeq, BrK, KN, Ledrap, StP}, using transfer
operators. The transfer operator, in combination with Markov
extensions, proved a powerful tool for so-called Collet-Eckmann
unimodal maps (see \eqref{eq:CE} below) for Keller and Nowicki
\cite{KN}, who showed that an appropriately weighted version of the
transfer operator is quasi-compact. To our knowledge, however,
these methods cannot be applied to non-Collet-Eckmann maps.

A less direct approach was taken by Pesin and Senti, results which were announced in \cite{PS2}, with details given in \cite{PS1}:
they used an inducing scheme $(X,F,\tau)$ (where $\tau$ is the
inducing time), a hyperbolic expanding
full branched map, albeit with infinitely many branches,
to find a unique equilibrium state $\mu_{\Phi_t}$
for the lifted potential $\Phi_t$. This equilibrium state is then projected
to the interval to give a measure $\mu_{\phi_t}$, a candidate equilibrium state for the system $(I,f,\phi_t)$.  It is proved that in the case where $f$ is a unimodal map
satisfying the strong exponential growth along critical orbits given in \cite{Se},
$\mu_{\phi_t}$ is a true equilibrium state for the whole system.  The down-side
for the more general case is that $\mu_{\phi_t}$ is only an `equilibrium state'
within the class of measures that are {\em compatible} to the inducing scheme, i.e., the induced map $F=f^\tau$ is defined for all iterates $\mu$-a.e. on $X$ and the inducing time $\tau$ is
$\mu_F$-integrable (here $\mu_F$ is the `lift' of $\mu$, see below).  A priori, the `equilibrium states' obtained in this way may not be true equilibrium states for the whole system, and different inducing schemes may lead
to different measures $\mu_{\phi_t}$.  Indeed, there exist measures with good properties which lift to some inducing schemes, but not to others: for example if $X$ is small then the set of points which never enter $X$ under iteration by $f$ can support measures of positive entropy.  Furthermore, inducing schemes are not always readily
available in general.

In this paper we show how to create `natural' inducing schemes and how to compare measures which `lift to' different schemes.

Our results are the first to deal with equilibrium states for the potential $\phi_t:x\mapsto -t\log|Df(x)|$ when $f$ is not Collet-Eckmann.   (We emphasise that the corresponding theory in \cite{PS1} considers a particular set of maps Collet-Eckmann maps  close to the Chebychev map.)  We also prove results on the analyticity of $t\mapsto P(\phi_t)$.

The Lyapunov exponent of a measure $\mu$ is defined as $\lambda(\mu) := \int_I
\log |Df|~d\mu$. Let $\M_{erg}$ be the set of all ergodic
$f$-invariant probability measures, and
\[
\M_+ = \left\{ \mu \in \M_{erg} :
\lambda(\mu) > 0,\ \supp(\mu) \not\subset
\orb(\Crit) \right\}.
\]
Measures $\mu$ with $\supp(\mu) \subset \orb(\Crit)$ are atomic.  Atomic measures in $\M_{erg}$ must be supported on periodic cycles.  So if  $\supp(\mu) \subset \orb(\Crit)$ and $\lambda(\mu)>0$, $\mu$ must be supported on a hyperbolic repelling periodic cycle, and thus the corresponding critical point must be preperiodic.  (Note that for $t \le 0$ such a situation can produce non-uniqueness of equilibrium states, see \cite{MSm1} and Section~\ref{sec:hypo}.)

\begin{theorem} \label{thm:poly}
Let $f \in \H$ be transitive with negative Schwarzian derivative
and let $\phi_t := -t\log |Df|$ for $t \in \R$.
Suppose that
for some $t_0 \in (0,1)$, $C > 0$ and $\beta > \ell_{max}(1+\frac1{t_0})
- 1$,
\begin{equation}\label{eq:polynomial}
|Df^n(f(c))| \ge C n^\beta \quad \mbox{ for all } c \in \Crit
\mbox{ and } n \ge 1.
\end{equation}
Then there exists $t_1 \in (t_0,1)$ \st the following hold:
\begin{itemize}
\item for every $t \in [t_1,1]$,
$(I,f,\phi_t)$ has an equilibrium state $\mu_{\phi_t} \in \M_+$;

\item
if $t_1 < t < 1$, then $\mu_{\phi_t}$ is the unique equilibrium state in $\M_{erg}$
and a compatible inducing scheme with respect to which $\mu_{\phi_t}$
has exponential tails;

\item if $t = 1$, then there may be other equilibrium states in
$\M_{erg} \sm \M_+$.   However, for $\mu_{\phi_1}\in \M_+$ there is a compatible inducing scheme with respect to which $\mu_{\phi_1}$ has polynomial tails;

\item the map $t \mapsto P(\phi_t)$ is analytic on
$(t_1,1)$.
\end{itemize}
\end{theorem}

We refer to this situation as the \emph{summable case}.
Note that for $t=1$ the measure $\mu_{\phi_1}\in \M_+$ is an absolutely continuous invariant measure (acip).  Therefore this result improves on the polynomial case of
\cite[Proposition 4.1]{BLS}, since in that theorem the polynomial decay of the tails was given under the above conditions, but also assuming that the critical points must all have the same order.  Results of \cite{BRSS} enable us to drop this assumption.  As was shown in \cite{BLS}, this tail decay rate implies that
the decay of correlations is at least polynomial.

As in the theorem, for $t = 1$ equilibrium states with zero
Lyapunov exponent are possible, see Section~\ref{sec:hypo} for details.
Let us explain why for $t < 1$, equilibrium states must
have $\lambda(\mu) > 0$.
The pressure function $t \mapsto P(\phi_t)$ is a continuous
decreasing function.  As in \cite{BRSS}, condition \eqref{eq:polynomial} implies the
existence of an acip  $\mu_1$ with $\lambda(\mu_1) > 0$, which is also a
equilibrium state for the potential $\phi_1 = -\log |Df|$. It
follows that
\begin{equation}\label{eq:pressure}
P(\phi_t) \ge (1-t) \lambda(\mu_1) \qquad \mbox{ for all } t \in
\R,
\end{equation}
so if $t< 1$ we have $P(\phi_t) > 0$.  By \cite{Prz}, we have $\lambda(\mu) \ge 0$ for any invariant measure, so Ruelle's inequality \cite{Ruelleineq} implies that $h_\mu(f) \le
\lambda(\mu)$.  Thus (for $t < 1$) equilibrium
states have positive Lyapunov exponent because
$\lambda(\mu) = 0$ implies $P(\phi_t) = 0$.

Notice that for $t \le 0$, the
potential $-t\log|Df|$ is upper semicontinuous, and the entropy
function $\mu\mapsto h_\mu(f)$ is upper semicontinuous, as
explained in \cite{Kbook}. This guarantees the existence of
equilibrium states for $(I,f)$ when $t \le 0$, regardless of
whether \eqref{eq:polynomial} holds or not.

A stronger condition than \eqref{eq:polynomial} is the
\emph{Collet-Eckmann condition} which states that there exist $C,
\alpha > 0$ \st
\begin{equation}\label{eq:CE}
|Df^n(f(c))| \ge C e^{\alpha n} \mbox{ for all } c \in \Crit \mbox{
and } n \in \N.
\end{equation}
This condition implies that $\lambda(\mu) > 0$
for every $\mu \in \M_{erg}$, see e.g.\ \cite{NSa} (and \cite{BS}
for the proof in the multimodal case).
In the unimodal case, the difference between Collet-Eckmann and
non-Collet-Eckmann maps can be seen from the behaviour of the
pressure function
at $t=1$, as follows from \cite{NSa}.  Indeed,
if \eqref{eq:polynomial} holds but not \eqref{eq:CE}, then there are
periodic orbits with Lyapunov exponents arbitrarily close to $0$,
and hence $P(\phi_t) = 0$ for $t \ge 1$. This is
regardless of the existence of equilibrium states, which, for $t >
1$, can only be measures for which $\lambda(\mu) = h_\mu(f) = 0$.
This means that the function $t\mapsto P(\phi_t)$
is not differentiable at $t=1$: we say that there is a
\emph{phase transition} at $1$.
See Section~\ref{sec:hypo} for more details on the phase transition,
and on maps without equilibrium states.

For unimodal Collet-Eckmann maps, the map $t\mapsto P(\phi_t)$ is
analytic in a neighbourhood of $1$, as was shown in \cite{BrK}.
The following theorem (the proof of which introduces many of the
ideas used for Theorem~\ref{thm:poly}) generalises this result
to all $f\in \H$ satisfying \eqref{eq:CE}, and gives results on equilibrium states also.

\begin{theorem}\label{thm:CE}
Suppose $f \in \H$ is transitive with negative Schwarzian derivative and $\phi_t = -t \log|Df|$.
If $f$ is Collet-Eckmann, then there exist $t_1<1<t_2$ \st $f$
has a unique equilibrium state $\mu_{\phi_t}$  for $t \in (t_1,t_2)$.
Moreover, $\mu_{\phi_t} \in \M_+$, there is a compatible inducing scheme
with respect to which $\mu_{\phi_t}$ has exponential tails,
and the map $t\mapsto P(\phi_t)$ is analytic in $(t_1,t_2)$.
\end{theorem}

In fact, the techniques used to prove this theorem also give
analyticity of the pressure for the special Collet-Eckmann maps
considered in \cite{PS1} for all $t$
in a neighbourhood of $[0,1]$.

{\bf Lifting measures.} Our main theorems deal with equilibrium states in $\M_+$. Although measures in $\M_+$ may not always be compatible to a specific inducing scheme, they are all compatible to some inducing scheme.  Given an inducing scheme $(X,F, \tau)$, we say that a measure $\mu_F$ is a \emph{lift} of $\mu$ if for all $\mu$-measurable subsets $A\subset I$,
\begin{equation} \mu(A) = \frac1{\int_X \tau \ d\mu_F} \sum_i \sum_{k = 0}^{\tau_i-1} \mu_F( X_i \cap f^{-k}(A)). \label{eq:lift}
\end{equation}
Conversely, given a measure $\mu_F$ for $(X,F)$, we say that
$\mu_F$ \emph{projects} to $\mu$ if \eqref{eq:lift} holds.

Let $X^\infty = \cap_n F^{-n}(\cup_i X_i)$ be the set of
points on which all iterates of $F$ are defined. The following
theorem gives us a method for finding inducing schemes, which are
naturally related to measures of positive Lyapunov exponent.

\begin{theorem}\label{thm:lyap lift}
If $\mu \in \M_+$, then there is an inducing scheme $(X,F,\tau)$
and a measure $\mu_F$ on $X$ \st $\int_X \tau \ d\mu_F < \infty$.
Here $\mu_F$ is the lifted measure of $\mu$ (i.e., $\mu$ and
$\mu_F$ are related by \eqref{eq:lift}).  Moreover, if $\Omega$
is the transitive component supporting $\mu$ then
$\overline{X^\infty}=X\cap\Omega$.

Conversely, if $(X,F,\tau)$ is an inducing scheme and $\mu_F$ an
ergodic $F$-invariant measure \st $\int_X \tau d\mu_F < \infty$, then
$\mu_F$ projects to a measure $\mu \in \M_{erg}$ with positive
Lyapunov exponent.
\end{theorem}

We would like to highlight another important set of results in this paper, which will be explained more fully later:
We will also show that all `relevant measures' in this paper lift to a fixed inducing scheme, see Proposition~\ref{prop:unif scheme} and Lemmas~\ref{lem:eq sees all} and \ref{lem:uniform scheme for nat}.

The potential $\phi_t$ (or $-t\log|Jf|$ in a wider setting,
where $Jf$ is the Jacobian of the map)
has geometric importance if $t$ is the dimension
of the phase space, because then the equilibrium state
can often be shown to be absolutely continuous with respect to $t$-dimensional
Hausdorff measure.
One can also consider other potentials: e.g.\ the seminal paper
by Bowen \cite{Bowen} applies to the class of \hold potentials.
In the setting  of interval maps, interesting results and examples were
given by Hofbauer and Keller \cite{HKeq} for potentials with bounded variation.
Our methods extend to such potentials as well.  We develop this theory in \cite{BTeqrange}.

The paper is organised as follows. Section~\ref{sec:prel} gives
preliminaries on (Gurevich) pressure, recurrence, and gives an
important result  on symbolic systems, due to Sarig. Also we
review basic results for interval maps. Section~\ref{sec:induce}
explains how to find inducing schemes using the Hofbauer tower,
which have the important property of being first return map on
this tower, even if the inducing scheme is not the first return on
the original system $(I,f)$. Theorem~\ref{thm:lyap lift} is proved
here as well. In Section~\ref{sec:our prelims} we prove
Proposition~\ref{prop:all_indu}, which gives the basic framework
of the existence and uniqueness proofs. Section~\ref{sec:poly}
is devoted to the main part of the proofs of Theorems~\ref{thm:poly} and \ref{thm:CE}  (using estimates from \cite{BLS}).
In Section~\ref{sec:tails}, we show that most equilibrium states in
this paper can be obtained from a Young tower with exponential
tails (see \cite{Y} for definitions), and discuss several
consequences of this remarkable fact, including the concluding
part of Theorems~\ref{thm:poly} and \ref{thm:CE}: the analyticity of the pressure function. Finally in Section~\ref{sec:hypo}, we discuss the hypotheses of our main theorems and give counter-examples that show that these
hypotheses cannot be easily relaxed.

{\em Acknowledgements:} We would like to thank Ian Melbourne,
Mariusz Urba\'nski, Peter Raith and Beno\^{\i}t Saussol for
fruitful discussions and comments on (earlier) versions of this paper.
We are especially grateful to Neil Dobbs whose remarks have led to
substantial clarification and strengthening of parts of this paper.

\section{Preliminaries}\label{sec:prel}

\subsection{Measures and Pressure}

If $(X,T)$ is a dynamical system with potential $\Phi:X \to \R$,
then the measure $m$ is {\em $\Phi$-conformal} if
\[
m(T(A)) = \int_A e^{-\Phi(x)}~dm(x)
\]
whenever $T:A \to T(A)$ is one-to-one. In other words, $dm\circ
T(x) = e^{-\Phi(x)}dm(x)$. We define the transfer operator for the
potential $\Phi$ as
\[
\L_{\Phi}g(y) := \sum_{T(y) = x} e^{\Phi(y)} g(y).
\]
We want to show that whatever inducing scheme we start with, the
invariant measure we get on $I$ is unique.  One of the key tools is
the following theorem which is the main result of \cite{SaBIP}.
Assume that $\S_1 = \{ X_i \}$ is a Markov partition of $X$ \st
$T:X_i \to X$ is injective for each $X_i \in \S_1$. We say that
$(X,T)$ has the \emph{big images and preimages (BIP)} property if,
there exist $X_1,\ldots, X_N\in \S_1$ \st for every $X_k \in \S_1$
there are $i,j \in \{ 1, \ldots, N\}$ and $x\in X_i$ \st $T(x) \in
X_k$ and $T^2(x) \in X_j$.

Suppose that $(X,T)$ is topologically mixing.  For every $X_i
\in \S_1$ and $n\ge 1$ let
\[
Z_n(\Phi, X_i):=\sum_{T^n x = x}e^{\Phi_n(x)}1_{X_i}(x),
\]   where $\Phi_n(x)= \sum_{j=0}^{n-1} \Phi \circ T^j(x)$.  Let
\[ Z_n^*(\Phi, X_i) := {\sum_{\stackrel{T^nx=x,}{T^kx \notin X_i \ \mbox{\tiny for}\ 0< k < n}} e^{\Phi_n(x)}} 1_{X_i}(x).
\]
We define the \emph{Gurevich pressure} of $\Phi$ as
\begin{equation}\label{eq:Gur}
P_G(\Phi) := \limsup_{n \to \infty}\frac 1n \log Z_n(\Phi, X_i).
\end{equation}
This limit exists, is independent of the choice of $X_i$ and it is
$> -\infty$, see \cite{Sathesis}. To simplify the notation, we
will often suppress the dependence of $Z_n(\Phi,X_i)$ and
$Z_n^*(\Phi,X_i)$ on $X_i$.  Furthermore, if $\|\L_\Phi
1\|_\infty<\infty$ then $P_G(\Phi)<\infty$, see Proposition 1 of
\cite{Sathesis}.

The potential $\Phi$ is said to be {\em recurrent} if
\begin{equation}
\sum_n \lambda^{-n} Z_n(\Phi) = \infty \mbox{ for } \lambda = \exp
P_G(\Phi).
\end{equation}
Moreover, $\Phi$ is called \emph{positive recurrent} if it is
recurrent and  $\sum_n n\lambda^{-n}Z^*_n(\Phi) = \infty$. We define
the {\em $n$-th variation} of $\Phi$ as
\begin{equation}\label{eq:var}
V_n(\Phi) := \sup_{\c_n \in \S_n} \sup_{x,y \in \S_n} |\Phi(x) -
\Phi(y)|,
\end{equation}
where $\S_n = \bigvee_{j=0}^{n-1} T^{-j}(\S_1)$ is the $n$-joint of
the Markov partition $\S_1$.

\begin{theorem}[\cite{SaBIP}] If $(X,T)$ is topologically mixing
and $\sum_{n\ge 1}V_n(\Phi)<\infty$, then $\Phi$ has an invariant
Gibbs measure if and only if $A$ has the BIP property and
$P_G(\Phi)<\infty$.  Moreover the Gibbs measure $\mu_\Phi$ has the
following properties \begin{itemize}
\item[(a)] If $h_{\mu_\Phi}(T) < \infty$ or $-\int \Phi d\mu_\Phi
< \infty$ then $\mu_\Phi$ is the unique equilibrium state (in
particular, $P(\Phi) = h_{\mu_\Phi}(T) + \int_X \Phi~d\mu_\Phi$);
\item[(b)] If $\|\L_\Phi 1\|_\infty<\infty$ then the Variational
Principle holds, i.e.,
$P_G(\Phi)=P(\Phi)$ ($=h_{\mu_\Phi}(T) + \int_X \Phi~d\mu_\Phi$);
\item[(c)] $\mu_\Phi$ is finite and $\mu_\Phi=\rho_\Phi~dm_\Phi$ where $\L_\Phi \rho_\Phi=\lambda \rho_\Phi$ and $\L_\Phi^*m_\Phi=\lambda m_\Phi$ for $\lambda=e^{P_G(\Phi)}$, i.e., $m_\Phi(TA) = \int_A e^{\Phi - \log \lambda}~dm_\Phi$;
\item[(d)]  This $\rho_\Phi$ is unique and $m_\Phi$ is the unique
$(\Phi-\log\lambda)$-conformal probability measure.
\end{itemize}
\label{thm:BIP}
\end{theorem}
Note that because $\mu_\Phi$ is a Gibbs measure, $\mu_\Phi(\c_n) >
0$ for every cylinder set $\c_n \in \S_n$, $n \in \N$.

In the paper of Mauldin \& Urba\'nski \cite{MU} several similar
results can be found, although they use a different approach to
pressure, taking the supremum of $\Phi_n$ on cylinder sets rather
than the value of $\Phi_n$ at periodic points.

\subsection{Interval Maps}

An interval map $(I,f)$ is called {\em piecewise monotone} if
there is a finite partition $\P_1$ into maximal intervals on which
$f$ is diffeomorphic. We call this partition the {\em
branch partition}.
We will assume that $f$ is $C^2$;
negative Schwarzian derivative in this $C^2$
context means that $1/\sqrt{|Df|}$ is a convex function on each $\c \in \P_1$.

\begin{remark}
The negative Schwarzian derivative condition allows us to use the Koebe lemma
for distortion control of the branches of
the induced maps we obtain later.  However if $f\in \H$ is
$C^3$ and there are no neutral periodic cycles,
then it is unnecessary to assume negative Schwarzian
derivative.  This was proved in the unimodal setting by Kozlovski
\cite{Koz}, and later for $f\in C^{2+\eta}$
in \cite{T}.  In the multimodal setting for $f\in C^3$ this was
proved by van Strien and Vargas \cite{SVarg}. \label{rmk:Sch}
\end{remark}

Let $\P_n = \bigvee_{k=0}^{n-1} f^{-k} \P_1$.  Elements $\c_n \in
\P_n$ are called {\em $n$-cylinders}. Similarly to \eqref{eq:var},
the \emph{$n$-th variation} of a potential $\phi:I \to \R$ is
defined as
\begin{equation*}\label{eq:var2}
V_n(\phi) = \sup_{\c_n \in \P_n} \sup_{x,y \in \c_n} |\phi(x) -
\phi(y)|.
\end{equation*}

The {\em non-wandering set } $\Omega$ of $f$ is the set of points
$x$ having arbitrarily small neighbourhoods $U$ \st $f^n(U) \cap U
\neq \emptyset$ for some $n \ge 1$.  Piecewise monotone $C^2$ maps
have non-wandering sets that split into a finite or countable number of {\em
transitive} components $\Omega_k$ \st each $\Omega_k$ contains a
dense orbit, see \cite{HR} and references therein.
A transitive component is one of the following:
\begin{itemize}
\item[($\Omega$1)] A finite union of intervals, cyclically permuted by $f$.
This is the most interesting case, and
Lemma~\ref{lem:trans_subgraph}(a) in Section~\ref{sec:induce} gives its description on the
Hofbauer tower.
\\
\item[($\Omega$2)] A Cantor set if $f$ is {\em infinitely renormalisable},
i.e, there is an infinite sequence of periodic intervals $J_n$
of increasing periods,
and $\Omega = \cap_n \orb(J_n)$.
Measures on such components have
$\lambda(\mu) = 0$, see \cite{MSbook} and  \cite[Theorem D]{SVarg}
for the multimodal case.
For maps that are only piecewise $C^2$, this
is no longer true, see Section~\ref{sec:hypo}.
\\
\item[($\Omega$3)] If $f$ is (finitely) renormalisable, say it has
a periodic interval $J \neq I$, then the set of
points that avoid $\orb(J)$ contains a transitive component
as well. This is usually a Cantor set, but it could be a finite set (e.g.\ if
$f$ is the Feigenbaum map).
For infinitely renormalisable maps, there are countably many
transitive components of this type.
Lemma~\ref{lem:trans_subgraph}(b)
in Section~\ref{sec:induce} gives its description
on the Hofbauer tower.
\end{itemize}
We will
state our results for transitive interval maps, but they can be
applied equally well to $(\Omega_k, f)$ for any component
$\Omega_k$ of the non-wandering set. In all our main theorems
we assume that $(\Omega, f)$ is {\em topological mixing}
(i.e., every iterate of $f$ is
topologically transitive).
This can be achieved by taking a transitive component of an appropriate
iterate of $f$.

We say that $(X,F,\tau)$ is an {\em inducing scheme} over $(I,f)$
if
\begin{itemize}
\item $X$ is a union of intervals containing a (countable)
collection of disjoint intervals $X_i$ \st $F$ maps each $X_i$
diffeomorphically onto $X$, with bounded distortion.
\item $F|_{X_i} = f^{\tau_i}$ for some $\tau_i \in \N := \{ 1,2,3 \dots \}$.
\end{itemize}
The function $\tau:\cup_i X_i \to \N$ defined by $\tau(x) =
\tau_i$ if $x \in X_i$ is called the {\em inducing time}. It may
happen that $\tau(x)$ is the first return time of $x$ to $X$, but
that is certainly not the general case.  For ease of notation, we will often let $(X,F,\tau)=(X,F)$.

Recall that $X^\infty = \cap_n F^{-n}(\cup_i X_i)$ is the set of
points on which all iterates of $F$ are defined. We call a measure
$\mu$  \emph{compatible}\label{p:compatible} to the inducing scheme if
\begin{itemize}
\item $\mu(X)> 0$ and $\mu(X \setminus X^{\infty}) = 0$, and
\item there exists a measure $\mu_F$ which projects to $\mu$ by
\eqref{eq:lift}, and in particular $\int_X \tau \ d\mu_F <
\infty$.
\end{itemize}

\begin{remark}
\begin{itemize}
\item[(a)] If $\mu\in \M_+$, applying Theorem~\ref{thm:lyap lift} gives us an inducing scheme $(X,F)$ and a measure $\mu_F$ satisfying the above conditions.

\item[(b)] $\overline{X^\infty}=X$ implies that given a measure
$\mu_F$ obtained from Theorem~\ref{thm:BIP}, the measure $\mu$,
the projection of $\mu_F$, has $\mu(U)>0$ for any open set in
$\cup_nf^n(X)$.

\item[(c)] If $(X,F,\tau)$ comes from Theorem~\ref{thm:lyap lift},
then $\mu$ is compatible to it if and only if $\mu(X^{\infty}) >
0$; for more general inducing schemes, this equivalence is false.

\item[(d)] Note that $\int\tau~d\mu<\infty$ does not always imply that $\int\tau~d\mu_F<\infty$, see \cite{Zwei}.
\end{itemize}\label{rmk:compat}
\end{remark}

The inducing scheme $(X,F)$ will perform the role of $(X,T)$ of the
previous section, with $\S_1 = \{ X_i\}$. Since $F$ maps $X_i$
onto $X$, the BIP property is automatically satisfied
provided $F$ is transitive (if not, we can always select a
transitive component). Let us denote the collection of
$n$-cylinders of the inducing scheme by $\S_n$. A priori, $\S_n$ is
not connected to $\cup_{m \ge 0} \P_m$, i.e., the cylinder sets of
the branch partition $\P_1$. In this paper, however, we will always
take $X$ to be a subset of $\cup_k \P_k$, and in that case the
$\cup_{n \ge 1} \S_n \subset \cup_{k \ge 1} \P_k$.

Given a potential $\phi:I \to \R$, let the {\em lifted potential}
$\Phi$ be defined by $\Phi(y) = \sum_{j=0}^{\tau_i-1} \phi \circ
f^j(y)$ for $y \in X_i$. We say that $\Phi$ has \emph{summable
variations} if $\sum_{n\ge 1} V_n(\Phi)<\infty$, and that $\Phi$ is
\emph{weakly H\"older continuous} if there exist $C_\Phi>0$ and
$0<\lambda_\Phi<0$ \st $V_n(\Phi) \le C_\Phi \lambda_\Phi^n$ for all
$n\ge 1$.  Clearly if $\Phi$ is weakly H\"older continuous then
$\Phi$ has summable variations.

We use summability of variations to control distortion of
$\Phi_n(x) = \Phi(x) + \dots + \Phi \circ F^{n-1}(x)$, but for the
potential $\phi_t = -t \log|Df|$, we can also use the Koebe Lemma
provided $f$ has negative Schwarzian derivative: If $X'\supset X$
\st $X'$ is a {\em $\delta$-scaled neighbourhood} of $X$, i.e.,
both components of $X' \setminus X$ have length $\ge \delta |X|$,
and $f^k:X_i \to X$ extends diffeomorphically to $f^k:X'_i \to
X'$, then
\[
\frac{|Df^k(y)|}{|Df^k(x)|} <\frac{1+2\delta}{\delta^2}+1
\]
for all $x,y \in X_i$.

In this paper we say $A_n\asymp B_n$ if $\lim_{n\to \infty}\frac{A_n}{B_n} =1$.  We will also say that $A\asymp_{dis} B$ if $A$ is equal to $B$ up to some distortion constant.

\section{Finding Inducing Schemes}\label{sec:induce}
In this section we will prove Theorem~\ref{thm:lyap lift}. The
idea relies on the construction of the {\em canonical Markov
extension} $(\hat I, \hat f)$ of the interval map. A measure $\mu
\in \M_+$ can be lifted to $(\hat I, \hat f)$, see
\cite{Kellift}, and in this space a first return map to a specific
subset $\hat X \subset \hat I$ gives rise to the inducing scheme.

The canonical Markov extension (commonly called {\em Hofbauer
tower}), was introduced by Hofbauer and Keller, see e.g.\ \cite{Htop,
Kellift}; it is a disjoint union of subintervals $D = f^n(\c_n)$,
$\c_n \in \P_n$, called {\em domains}, where $\P_1$ is the branch
partition. Let $\D$ be the collection of all such domains. For
completeness, let $\P_0$ denote the partition of $I$ consisting of
the single set $I$, and call $D_0 = f^0(I)$ the {\em base} of the
Hofbauer tower. Then
\[
\hat I = \sqcup_{n \ge 0} \sqcup_{\c_n \in \P_n} f^n(\c_n) / \sim,
\]
where $f^n(\c_n) \sim f^m(\c_m)$ if they represent the same
interval. Let $\pi: \hat I \to I$ be the inclusion map. Points
$\hat x\in \hat I$ can be written as $(x,D)$ if $D$ is the domain
that $\hat x$ belongs to and $x = \pi(\hat x)$. The map $\hat
f:\hat I \to \hat I$ is defined as
\[
\hat f(\hat x) = \hat f(x,D) = (f(x), D')
\]
if there are cylinder sets $\c_n \supset \c_{n+1}$ \st $x \in
f^n(\c_{n+1}) \subset f^n(\c_n) = D$ and $D' = f^{n+1}(\c_{n+1})$.
In this case, we write $D \to D'$, giving $(\D, \to)$ the
structure of a directed graph. It is easy to check that there is a
one-to-one correspondence between cylinder sets $\c_n \in \P_n$
and $n$-paths $D_0 \to \dots \to D_n$ starting at the base of the
Hofbauer tower. For each $R \in \N$, let $\hat I_R$ be the compact
part of the Hofbauer tower defined by
\[
\hat I_R = \bigcup \{ D \in \D : \mbox{ there exists a path } D_0
\to \dots \to D \mbox{ of length } r \le R \}
\]
A subgraph $(\E, \to)$ is called \emph{closed} if $D \in \E$ and
$D \to D'$ implies that $D' \in \E$. It is \emph{primitive} if for
every pair $D, D' \in \E$, there is a path from $D$ to $D'$ within
$\E$. Clearly any two distinct maximal primitive subgraphs are
disjoint.

\begin{lemma}\label{lem:trans_subgraph}
Let $f:I\to I$ be a multimodal map and $\Omega$ is a transitive component.\\
(a) If $\Omega$ consists of a finite union of intervals,
then there is a closed primitive subgraph
$(\E, \to)$ of $(\D, \to)$ containing a dense $\hat f$-orbit and
\st $\Omega = \pi(\cup_{D \in \E} D)$.\\
(b)  If $\Omega$ is a Cantor (or finite) set avoiding a periodic
interval of $J$, then there is a (non-closed) primitive subgraph
$(\E, \to)$ of $(\D, \to)$
\st $\Omega \subset \pi(\cup_{D \in \E} D)$,
and there is a dense $\hat f$-orbit in
$(\cup_{D \in \E} D) \cap \pi^{-1}(\Omega)$.
\end{lemma}

The arguments for this lemma are implicit in \cite{Htop,HR}
combined. We will give a self-contained proof in the appendix.
Notice that $(\hat I, \hat f)$ is a Markov map in the sense that
the image of any domain $D$ is the union of domains of $\hat I$.
Obviously, $\pi \circ \hat f = f \circ \pi$.

Recall that $D_0 = I = f^0(\c_0)$ is the base of the Hofbauer
tower. Let $i:I \to D_0$ be the trivial bijection map (inclusion)
\st $i^{-1} = \pi|_{D_0}$. Given a measure $\mu \in \M_{erg}$, let
$\hat \mu_0 = \mu \circ i^{-1}$, and
\begin{equation}
\hat{\mu}_n := \frac{1}{n}\sum_{k=0}^{n-1} \hat \mu_0
\circ\hat{f}^{-k}. \label{eqn:mulift}
\end{equation}
We say that $\mu$ is {\em liftable} to $(\hat I, \hat f)$ if there
exists a weak accumulation point $\hat \mu$ of the sequence
$\{\hat\mu_n\}_n$ with $\hat\mu\not\equiv 0$.

\begin{remark}\label{rmk:lifted_ergodic}
If $\mu$ is liftable and ergodic, then $\hat\mu$ is an ergodic
$\hat f$-invariant probability measure on $\hat I$, see \cite{Kellift}
\end{remark}

\begin{proof}[Proof of Theorem~\ref{thm:lyap lift}]
First assume that $\mu \in \M_+$. Keller \cite{Kellift} showed
that if $\mu$ is not atomic then it is liftable, $\hat \mu(\hat I) = \mu(I) = 1$ and $\hat\mu\circ\pi^{-1} =\mu$.  If $\mu\in \M_+$ is atomic, it must be supported on a hyperbolic repelling periodic cycle.  It is easy to show that such measures are liftable.  In both cases, \cite{Kellift} shows that $\hat\mu$ is also ergodic.

Now take some domain $D$ and
cylinder set $\c_n \in \P_n$ \st $\pi(D)$ compactly contains
$\c_n$ and $\hat \mu(\hat X) > 0$ for $\hat X := \pi^{-1}(\c_n)
\cap D$. Let $\hat F:\hat X \to \hat X$ be the first return map;
let $\hat\tau(x) \in \N$ be \st $\hat F(x) = \hat f^{\hat\tau(\hat
x)}(\hat x)$ for each $\hat x \in \hat X$ on which $\hat F$ is
defined. By the Markov property of $\hat f$, $\hat x$ has a
neighbourhood $U$ \st $\hat f^{\hat\tau(\hat x)}$ maps $U$
monotonically onto $D$. Therefore there is a neighbourhood $V
\subset U$ \st $\hat f^{\hat\tau(\hat x)}$ maps $V$ monotonically
onto $\hat X$. Since $\pi(\hat X) = \c_n$ is a cylinder set,
$\orb(\partial \hat X) \cap \hat X = \emptyset$. It follows that
$\hat\tau(\hat y) = \hat\tau(\hat x)$ for all $\hat y \in V$.

Let $\Omega$ be the transitive component supporting $\mu$.
If $\Omega$
is an interval as in case ($\Omega$1), then we take $D$ inside the closed transitive subgraph
of $(\D, \to)$ as guaranteed by Lemma~\ref{lem:trans_subgraph}(a).
Take any open interval $U \subset X$. Since $\P_1$ generates the
Borel $\sigma$-algebra there is an $n$-cylinder $\c_n \subset
U$; we let $\hat \c_n = \pi^{-1}(\c_n) \cap D$. It follows that
$\hat f^n(\hat \c_n) = D'$ for some domain $D'$ in the same
transitive component of the Hofbauer tower as $D$. Hence there is
an $m$-path $D' \to \dots \to D$ and a subcylinder $\hat \c_{n+m}
\subset \hat \c_n$ \st $\hat f^{n+m}(\hat \c_{n+m}) = D$.
Therefore $\pi(\hat \c_{n+m}) \subset U$ contains a domain $X_i$.
It follows that $\cup_i X_i$ is dense in $X$. Repeating the
argument for $U \subset X_i$ we find that $F^{-1}(\cup_i X_i)$ is
dense in $X$, and by induction, $X^{\infty}$ is dense in $X$ as
well. (Notice that this construction may produce many branches $X_i$
such that $\mu(X_i) = 0$, but this doesn't affect the result.)

If $\Omega$ is as in case ($\Omega$2) then $\M_+ =\es$ so there is nothing to show.  This is proved for the unimodal case in \cite{MSbook}; the multimodal case is similar, the required `real bounds' follow from \cite{SVarg}.
If $\Omega$ is Cantor (or finite) set of points avoiding a periodic
interval of $f$ as in case ($\Omega$3), then Lemma~\ref{lem:trans_subgraph}(b)
still provides us with a primitive subgraph, and the same
argument as above shows that $X^\infty$ is dense in $X \cap \Omega$.

Now the inducing scheme $(X, F, \tau)$ is defined by $X = \pi(\hat
X)$, $F = \pi\circ\hat F \circ \pi^{-1}|_{\hat X}$ and $\tau(x) = \hat
\tau(\pi^{-1}(x) \cap \hat X)$. Because $\mu = \hat\mu \circ
\pi^{-1}$, $\mu(X) \ge \hat \mu(\hat X) > 0$.

Let $\hat\mu_{\hat X}:=\frac1{\hat \mu(\hat X)}\hat\mu|_{\hat X}$  be
the conditional measure on $\hat X$.
The measure $\mu_F := \hat \mu_{\hat X} \circ  \pi^{-1}|_{\hat X}$
is clearly $F$-invariant, and by Kac's Lemma,
\[
\int_X \tau \ d\mu_F = \int_{\hat X} \hat \tau \ d\hat\mu_{\hat X}
= \frac{1}{\hat\mu (\hat X)} < \infty.
\]
Finally, by the Poincar\'e Recurrence Theorem, $\hat\mu_{\hat
X}$-a.e. point $\hat x \in \hat X$ returns infinitely often to
$\hat X$, and because $\mu_F \ll \mu $ we also get
$\mu(X^{\infty}) = \mu(X)$ by ergodicity.

Now for the other direction, notice that by assumption, each
branch of any iterate $F^n$ of the induced map has negative
Schwarzian derivative. Therefore distortion is bounded uniformly
over $n$ and the branches of $F^n$. Hence, by taking an iterate of
the induced map $F$ if necessary, we can assume that $F^n$ is
uniformly expanding. It follows by $F$-invariance of $\mu_F$ that
\begin{align*}
0 &< \frac1n \int_{X^{\infty}} \log |DF^n| \ d\mu_F \\
&= \int_{X^{\infty}} \frac1n \sum_{j=0}^{n-1} \log |DF \circ F^j|
\ d\mu_F = \int_{X^{\infty}} \log |DF| \ d\mu_F = \lambda(\mu_F).
\end{align*}
Let $\mu$ be the projected measure of $\mu_F$; both $\mu_F$ and
$\mu$ are ergodic. Since $\int \tau \ d\mu_F < \infty$, we can
take a point $x \in X^{\infty}$ which is typical for both $\mu_F$
and $\mu$. Let $\tau_k = \sum_{j=0}^{k-1} \tau \circ F^j(x)$. Then
applying the Ergodic Theorem several times, we get $\lim_{k \to
\infty} \frac{\tau_k}{k} = \int \tau d \mu_F < \infty$ and
\begin{align*}
\lambda(\mu) &= \int_I \log |Df| \ d\mu
= \lim_{n \to \infty} \frac1n \sum_{j=0}^{n-1} \log|Df \circ f^j(x)| \\
&=  \lim_{k \to \infty} \frac1{\tau_k}
\sum_{j=0}^{\tau_k-1} \log|Df \circ f^j(x)| \\
&=  \lim_{k \to \infty} \frac{k}{\tau_k} \ \frac1k
\sum_{j=0}^{k-1} \log|DF \circ F^j(x)| = \frac1{\int \tau d\mu_F}
\lambda(\mu_F) > 0.
\end{align*}
This concludes the proof.
\end{proof}

\begin{remark}
If $\lambda(\mu) > 0$ but $\supp(\mu) \subset \orb(\Crit)$ and
$\mu$ is the equidistribution on a repelling periodic orbit, say
$\supp(\mu) = \orb(p)$ where $f^n(p) = p$, then we can still find an inducing scheme compatible to $\mu$. Let $X \owns p$ be an
open interval \st the component of $f^{-n}(X)$ containing $p$ is
compactly contained in $X$. Call this component $X_1$. Then
$(X,F,\tau)$ with $F|_{X_1} = f^{\tau_1}|_{X_1} = f^n|_{X_1}$ is
an inducing scheme compatible to $\mu$.
\end{remark}

\begin{remark}
If $\mu\in \M_+$ then
Remark~\ref{rmk:lifted_ergodic} implies that $\hat\mu$ is ergodic.
If $\Omega$ is as in Lemma~\ref{lem:trans_subgraph}(a) we also have that $\hat\mu$ is supported on $\E$.  That lemma implies that for any $\hat x\in \hat I\sm\bd\D $ there is $\hat y\in \E$ so that $\pi(\hat x)=\pi(\hat y)$.
Thus there exists $n \ge 0$ so that $\hat f^n(\hat x)=\hat f^n(\hat y)$.   So $\hat\mu(\E)=1$ follows by ergodicity.  \label{rmk:E full meas}
\end{remark}

The induced system used in this proof may be the simplest but not
always the most convenient. Let us call an inducing scheme $(X, F,
\tau)$ a {\em first extendible return} scheme with respect to a
neighbourhood $Y$ of $X$ if for each $x \in X_i$, $\tau(x)$ is the
smallest positive iterate \st $f^j(x) \in X$ and there is a
neighbourhood $Y_i \supset X_i$ \st $f^j$ maps $Y_i$ monotonically
onto $Y$. If $Y$ is a fixed $\delta$-scaled neighbourhood $Y$,
then the Koebe Lemma can be used to control distortion of branches
of (iterates of) $F$.  In this case we say that $\tau$ is the \emph{first $\delta$-extendible return time to $X$}.

\begin{lemma}\label{lem:exten}
If $\mu\in \M_+$ then there exists $\delta>0$ and an interval $X\subset I$ such that $\mu$ is compatible to the inducing scheme $(X,F,\tau)$ where $\tau$ is the first $\delta$-extendible return time.  Moreover, if $\Omega$ is the transitive component supporting $\mu$ then $\overline{X^\infty} = X\cap \Omega$.
\end{lemma}

The proof of the first part of this lemma can be found in \cite{BrCMP}, but some of the ideas of the proof are particularly useful in this paper so we sketch those parts here.

\begin{proof} As we noted in the proof of Theorem~\ref{thm:lyap lift}, since $\mu\in \M_+$, $\hat\mu(\hat I)>0$.  We choose $X$ and $\delta>0$ so that the set
$\hat X = \sqcup \{ D \cap \pi^{-1}(X) : D \in \D, \pi(D) \supset
Y \}$, where $Y$ is concentric with $X$ and size $(1+2\delta)|X|$, has $\hat\mu(\hat X)>0$.  Let $r_{\hat X}$ denote the first return map to $\hat X$.  In \cite{BrCMP} it is shown that given $x\in X^\infty$, for any $\hat x\in \hat X$ with $\pi(\hat x)=x$, we have $r_{\hat X}(\hat x)=\tau(x)$.  As in \cite{BrCMP}, this can be used to prove that $\mu$ is compatible to $(X,F, \tau)$.

The proof that $\overline{X^\infty} = X\cap \Omega$ follows as in the proof of Theorem~\ref{thm:lyap lift}.
\end{proof}

Theorem~\ref{thm:lyap lift} exploits the fact that measures with
positive Lyapunov exponents are liftable; but their lifts do not,
in general, give similar mass to the same parts in the Hofbauer
tower. The next result shows that measures with entropy uniformly
bounded away from $0$ lift, and give mass uniformly to specific
compact subsets of the Hofbauer tower. The proof is postponed to
the appendix.

\begin{lemma}\label{lem:compat entro}
For every $\eps > 0$, there are $R \in \N$ and $\eta > 0$ such
that if $\mu \in \M_{erg}$ has entropy $h_\mu(f) \ge \eps$, then
$\mu$ is liftable to the Hofbauer tower and $\hat \mu(\hat I_R)
\ge \eta$. Furthermore, there is a set $\hat E$, depending only on
$\eps$, \st $\hat \mu(\hat E)> \eta/2$ and $\min_{D\in \D\cap\hat
I_R}d(\hat E\cap D, \bd D)>0$.
\end{lemma}

One consequence of this lemma is that the choice of $\delta$ in Lemma~\ref{lem:exten} depends only on the entropy of $\mu$.

Notice that by Remark~\ref{rmk:E full meas}, we can suppose that $\hat E\subset \E$.
We will use this lemma in connection with Case 4 of
Proposition~\ref{prop:all_indu} in the next section to carry out
the proofs of Theorems~\ref{thm:CE} and \ref{thm:poly}. In
principle, these results deal with measures in $\M_+$ that
possibly have zero entropy. However, the next lemma shows that our
equilibrium states need to have both positive Lyapunov exponent
and entropy.

\begin{lemma}\label{lem:free energy bound}
Suppose that $f\in \H$ satisfies \eqref{eq:polynomial}.
Then there exists $\zeta_1<0$ so that for $t\in (\zeta_1, 1)$, there exist $\eps_0, \, \eps>0$ so that any
measure $\nu$ with $h_\nu(f) + \int \phi_t~d\nu > P(\phi_t) -\eps_0$
satisfies $h_{\nu}(f)\ge \eps$.  Similarly, if
 $f\in \H$ satisfies \eqref{eq:CE} then there exist $\zeta_1<0<\zeta_2$ so that for $t\in (\zeta_1, 1+\zeta_2)$, there exist $\eps_0, \, \eps>0$ so that any
measure $\nu$ with $h_\nu(f) + \int \phi_t~d\nu > P(\phi_t) -\eps_0$
satisfies $h_{\nu}(f)\ge \eps$.
\end{lemma}

\begin{proof}
Any transitive map satisfying \eqref{eq:polynomial}
has an acip $\mu$ with $h_\mu(f) = \lambda(\mu)> 0$. Applying
\eqref{eq:pressure} and Ruelle's inequality \cite{Ruelleineq},
we obtain that $P(\phi_t)> 0$ for $t < 1$.  We let $\eps_0=\eps_0(t):={P(\phi_t)}/2$.  Therefore, it is easy to see that for all $t\in [0,1)$ there
exists $\eps = \eps(t) > 0$ \st $h_\nu(f) + \int \phi_t~d\nu > P(\phi_t)/2$
implies $h_{\nu}(f)>\eps$.  For the case $t<0$, let $\zeta_1:=-\frac{\htop(f)}{4\sup\{\lambda(\nu):\nu\in \M_{erg}\}}$.  Then $h_\nu(f) + \int \phi_t~d\nu > P(\phi_t)/2$ implies $h_\nu(f)> P(\phi_t)/2-t\lambda(\nu)$.  Since $P(\phi_t)> \htop(f)$, for $t\in (\zeta_1, 0)$ we obtain $h_\nu(f)>\htop(f)/4$.

Next assume that the Collet-Eckmann condition \eqref{eq:CE} holds.  We can choose $\zeta_1$ as above.
Define $\underline\lambda:=
\inf\{\lambda(\nu):\nu\in \M_{erg}\}$, and let
$\gamma:= \underline\lambda / \lambda(\mu) \le 1$.
By \cite[Theorem 1.2]{BS} we know that $\underline\lambda>0$.
Take $\eps = \underline{\lambda}/2$.
If $\nu$ is any measure with $h_\nu(f) < \eps$ then
\[
P(\phi_t) - \left(h_\nu(f) + \int \phi_t d\nu\right) \ge
\left[ (1-t) - \left(\frac12 - t\right)\gamma \right] \lambda(\mu)
= \left[ 1-\frac{\gamma}{2} +  t(\gamma-1) \right] \lambda(\mu),
\]
which is bounded away from $0$ for all fixed
$1 \le t < \frac{1-\gamma/2}{1-\gamma}$ (or all $t \ge 1$ if $\gamma = 1$).
Hence, if $h_{\nu}(f) < \eps$, then the free energy of $\nu$
cannot be close to $P(\phi_t)$.
\end{proof}

We are now able to state the following, which relates to part (c)
of Proposition~\ref{prop:all_indu}.

\begin{corollary}
In the setting of Theorems~\ref{thm:poly} and \ref{thm:CE},
there exists $\eta'>0$, a sequence $\{\mu_n\}_n$ \st
$h_{\mu_n}(f)+\int\phi_t~d\mu_n\to P(\phi_t)$ and an inducing
scheme $(X,F)$ given by Theorem~\ref{thm:lyap lift} or a
first extendible return map (as in Lemma~\ref{lem:exten}) \st $\hat\mu_n(\hat X)>\eta'$ for all $n$.
\label{cor:uniform X}
\end{corollary}

\begin{proof}
 From the definition of pressure, there exists $\{\mu_n\}\subset
\M_{erg}$ so that $h_{\mu_n}(f)+\int\phi_t~d\mu_n \to P(\phi_t)$.  By
Lemma~\ref{lem:free energy bound}, there exists $\eps>0$ so that
$h_{\mu_n}(f)\ge \eps$ for all large $n$.  Let $\hat E=\hat E(\eps)$ as
in Lemma~\ref{lem:compat entro}.
Firstly, for the
type of inducing scheme given by Theorem~\ref{thm:lyap lift},
there must exist $\eta'>0$, $D\in \D\cap\hat I_R$, a subset $\hat
E' \subset \hat E\cap D$ with $\pi(\hat E') \in \P_n$ and a
subsequence $n_k \to \infty$ \st $\mu_{n_k}(\hat E')\ge \eta'$.
Then we let $\hat E'$ be the inducing domain $\hat X$ in
Theorem~\ref{thm:lyap lift}.
Lemmas~\ref{lem:compat entro} and \ref{lem:free energy bound}
complete the proof.

For a first extendible inducing scheme as in Lemma~\ref{lem:exten},
the proof follows similarly.  The main point is to notice that the
set $\hat E$ from Lemma~\ref{lem:compat entro} has $\min_{D\in
\D\cap\hat I_R}d(\hat E\cap D, \bd D)>0$.
\end{proof}

\section{A Key Result for Existence and Uniqueness}
\label{sec:our prelims}
The proof of Theorem~\ref{thm:poly} is divided into several steps. We use the Hofbauer tower construction given in Section~\ref{sec:induce} to fix an inducing scheme $F:\bigcup_jX_j \to X$ over $X\in \P_n$. Let $\Phi$ be the induced potential.

The following lemma, the ideas for which go back to Abramov  \cite{Ab}, relates the free energies of the original and the induced system.  See \cite{PS1} for the proof.

\begin{lemma}
If $\mu_F$ is an ergodic measure on $(X,F)$ with $\int \tau d\mu_F
< \infty$, and $\mu$ is the projected measure on $(X,f)$, then
\[
h_{\mu_F}(F) =\left(\int_X\tau~d\mu_F \right) h_{\mu}(f) \hbox{ and }
\int_X\Phi~d\mu_F = \left(\int_X\tau~d\mu_F
\right)\int_I\phi~d\mu.
\]
where $\Phi$ is the lifted potential of $\phi$.
\label{lem:Abra}\end{lemma}

It is easy to show that putting $\phi:=\log|Df|$ into the above
lemma proves that for any full-branched inducing scheme with
ergodic invariant measure $\mu_F$, the measure projects to a
measure $\mu$ with $\lambda(\mu)>0$.

Suppose that $\phi:I\to \R$ is the potential for the original system.  We will deal with the shifted potential $\psi_S:=\phi-S$.  Given an inducing scheme $(X,F)$ with $F=f^\tau$, let $\Psi_S$ be the induced potential, i.e., $\Psi_S:=\Phi-\tau S$.  The following lemma resembles
the argument of \cite[Proposition 10]{Sathesis}.  An important difference
here is that we do not require that the original potential has summable variations.

\begin{lemma}
Suppose that $P_G(\Psi_{S^*})<\infty$ and $\Phi$ has summable variations.  Then $P_G(\Psi_S)$ is decreasing and continuous in $[S^*, \infty)$. \label{lem:press cts}
\end{lemma}

\begin{proof}
We first recall some facts.  By definition, $P_G(\Psi_S):=\lim_{n\to \infty}\frac1n \log Z_n(\Psi_S,X_i)$ where $Z_n(\Psi_S,X_i):=\sum_{F^nx=x}e^{(\Psi_{S})_n(x)}1_{X_i}= \sum_{F^nx=x}e^{\Phi_n(x)-S\tau^n(x)}1_{X_i}$.
As in \cite{Sathesis}, topological mixing implies that $P_G(\Psi_S)$ is independent of $X_i$,
and we suppress $X_i$ in the notation accordingly.  Clearly, $P_G(\Psi_S)$
is decreasing in $S$.  We also know that since we have summable variations for $\Phi$,
i.e., there exists $B<\infty$ \st $\sum_{k=1}^\infty V_n(\Phi)<B$,
we have for any $S$,
\begin{equation}
\log Z_{m_1}(\Psi_S) +\log Z_{m_2}(\Psi_S)\le \log Z_{m_1+m_2}(\Psi_S)+\log B,
\label{eq:subadd}
\end{equation} see the proof of \cite[Proposition 1]{Sathesis}.

Since $P_G(\Psi_S)$ is decreasing in $S$, it is sufficient to show that for any $S_0\ge S^*$ and any $\eps>0$, there exists $S>S_0$ \st $P_G(\Psi_S)> P_G(\Psi_{S_0})-\eps$.  Fix $\eps>0$ and $n_0$ so large that
$\frac{\log B}{n_0}<\frac\eps 3$.
By definition of $P_G(\Psi_{S_0})$, for a large enough $n \ge n_0$,
$$
\frac1{n}\log Z_{n}(\Psi_{S_0})\ge P_G(\Psi_{S_0})-\frac\eps3.
$$
Since $Z_{n}(\Psi_{S})$ is continuous in $S$, there exists $S>S_0$ \st
$$
\frac1{n}\log Z_{n}(\Psi_{S})>P_G(\Psi_{S_0})-\frac23\eps.
$$
Then by \eqref{eq:subadd} and writing $m=kn+r$ where $0\le r\le n-1$,
\begin{align*}
\frac{\log Z_m(\Psi_S)}m & \ge \frac{k \log Z_n(\Psi_S)+\log Z_r(\Psi_S)-(k+1)\log B}{kn+r}\\
& \stackrel{m\to \infty}{\longrightarrow}
\frac{\log Z_n(\Psi_S)}n-\frac{\log B}n \ge P_G(\Psi_{S_0})-\eps
\end{align*} as required.
\end{proof}

The following result is a key tool in proving
Theorems~\ref{thm:poly} and \ref{thm:CE}. It gives necessary
conditions, comparable to the abstract conditions presented in
\cite{PS1}, to push equilibrium states through inducing
procedures.  Notice that Case 4 is reminiscent of the ideas
involved in the Discriminant Theorem, \cite[Theorem 2]{Saphase}.
However, our approach seems more natural in this context.

\begin{proposition}\label{prop:all_indu}
Suppose that $\psi$ is a potential with $P(\psi)=0$.  Let $\hat X$ be
the set used in either Theorem~\ref{thm:lyap lift} or
Lemma~\ref{lem:exten} to construct the corresponding inducing scheme
$(X,F,\tau)$.  Suppose that the lifted potential $\Psi$ has
$\|\L_\Psi 1\|_\infty<\infty$ and $\sum_{n\ge 1}V_n(\Psi)<\infty$.

Consider the assumptions:
\begin{itemize}
\item[(a)] $\sum_i\tau_i e^{\Psi_i}<\infty$ for
$\Psi_i = \sup_{x \in X_i} \Psi(x)$;
\item[(b)] there exists an equilibrium state
$\mu\in \M_+$ compatible to $(X,F,\tau)$;
\item[(c)] there exist a sequence
$\{\eps_n\}_n\subset \R^-$ with
$\eps_n \to 0$ and measures
$\{\mu_n\}_n\subset \M_+$ \st every $\mu_n$ is compatible to
$(X,F,\tau)$,
$h_{\mu_n}(f)+\int\psi~d\mu_n = \eps_n$ and
$P_G(\Psi_{\eps_n})<\infty$ for all $n$;
\item[(d)] $P_G(\Psi)=0$.
\end{itemize}
If any of the following combinations of assumptions holds:
\[
\left\{
\begin{array}{ll}
1. & \mbox{(b) and (d)}; \\
2. & \mbox{(a) and (d)}; \\
3. & \mbox{(a) and (b)}; \\
4. & \mbox{(a) and (c)};
\end{array} \right.
\]
then there is a unique equilibrium state $\mu$ for $(I,f,\psi)$
among measures $\mu \in \M_+$ with $\hat\mu(\hat X)>0$. Moreover,
$\mu$ is obtained by projecting the equilibrium state $\mu_\Psi$ of
the inducing scheme and in all cases we have $P_G(\Psi)=0$.
\end{proposition}

\begin{remark}
As noted in the proof, if $\mu_\Psi$ is the equilibrium state for
$(X,F,\Psi)$ given by Theorem~\ref{thm:BIP} then the condition
$\sum_i\tau_i e^{\Psi_i} < \infty$ implies that $\int_Y \tau
d\mu_\Psi < \infty$ by the Gibbs property of $\mu_\Psi$.
\end{remark}

\begin{proof}[Proof of Proposition~\ref{prop:all_indu}]
As in Section~\ref{sec:prel}, Proposition 1 of
\cite{Sathesis} implies that $Z_n(\Psi)=O(\|\L_\Psi 1
\|_\infty^n)$.  Therefore $\|\L_\Psi 1
\|_\infty<\infty$ implies $P_G(\Psi)<\infty$.  So in any case we can
immediately apply Theorem~\ref{thm:BIP} to obtain a measure
$\mu_\Psi$, and moreover the Variational Principle holds.

{\bf Case 1. (b) and (d) hold:} By definition of compatibility, we
can lift $\mu$ to $\mu_F$ where $\int \tau~d\mu_F<\infty$. By
Lemma~\ref{lem:Abra} we have
$$
0=P(\psi)=\left(\int\tau~d\mu_F\right)
\left(h_\mu(f)+\int\psi~d\mu\right) =
h_{\mu_F}(F)+\int\Psi~d\mu_F.
$$
Since we also have $P_G(\Psi)=0$,
the Variational Principle (Theorem~\ref{thm:BIP} (b)) implies that
$\mu_F$ is an equilibrium state for the inducing scheme. From the
uniqueness of the measure given by Theorem~\ref{thm:BIP}, we have
$\mu_F=\mu_\Psi$.  So $\mu$ is the same as the projection of
$\mu_\Psi$ given by Theorem~\ref{thm:lyap lift}, as required. Note
that by Lemma~\ref{lem:Abra}, $h_{\mu_\Psi}(F)<\infty$ and
$-\int\Psi~d\mu_\Psi <\infty$.

{\bf Case 2: (a) and (d) hold:}  By the Gibbs property of
$\mu_\Psi$ we have
$$
\int\tau~d\mu_\Psi\asymp_{dis} \sum_i\tau_i e^{\Psi_i-P_G(\Psi)}<\infty.
$$
This implies that we can use Theorem~\ref{thm:lyap lift} to
project $\mu_\Psi$ to an $f$-invariant measure $\mu_\psi\in
\M_+$. By Lemma~\ref{lem:Abra}, $h_{\mu_\Psi}(F)<\infty$ and
$-\int\Psi~d\mu_\Psi <\infty$.
So by Theorem~\ref{thm:BIP} part (a), $\mu_\Psi$ is an equilibrium,
and the Variational Principle (i.e.,  Theorem~\ref{thm:BIP} part (b))
we have  $P_G(\Psi)=P(\Psi) = h_{\mu_\Psi}(F) +\int\Psi~d\mu_\Psi$.

Now condition (d) gives that $P_G(\Psi)=P(\Psi) = 0$.
Thus Lemma~\ref{lem:Abra} implies that
$h_{\mu_\psi}(f)+\int\psi~d\mu_\psi = 0$, so $\mu_\psi$ is an
equilibrium state. We can then use the argument of Case 1 to show that
this is the unique equilibrium state in $\M_+$ with $\hat \mu(\hat X)
= (\int \tau \ d\hat\mu)^{-1} > 0$.

{\bf Case 3: (a) and (b) hold:} We start as in Case 2; condition (a)
gives a measure
$\mu_\psi$ having $h_{\mu_\psi}(f) +\int\psi~d\mu_\psi \le
P(\psi)=0$.  By Lemma~\ref{lem:Abra} and the Variational Principle
this implies $P_G(\Psi)\le 0$.

Assumption (b) gives an equilibrium state $\mu \in \M_+$ which can be
lifted, using Theorem~\ref{thm:lyap lift}, to $\mu_F$ on $(X,F,\tau)$.
Now since we also have $0= h_{\mu}(f) +\int\psi~d\mu$, Lemma~\ref{lem:Abra}
implies that $0 \le \int \tau \ d\mu_F (h_{\mu}(f) +\int\psi~d\mu)
\le P(\Psi)$ and by the Variational Principle,
$0 \le P_G(\Psi)$ as well.
Thus we have $P_G(\Psi)=0$ and we can apply the argument of Case 1.

{\bf Case 4: (a) and (c) hold:} By the argument of Case 2 we have
an equilibrium state $\mu_\psi$.  Therefore, if we can show that
$P_G(\Psi)=0$, Case 1 above completes the proof.

The argument for Case 3 showed that $P_G(\Psi)\le 0$.
By (c), $h_{\mu_n}(f)+\int (\psi-\eps_n)~d\mu_n = -\eps_n > 0$. Let
$\mu_{n,F}$ be the corresponding lifted measure obtained from
Theorem~\ref{thm:lyap lift}. Then by Lemma~\ref{lem:Abra},
$0 \le h_{\mu_{n,F}}(F)+\int_X \Psi_{\eps_n}~d\mu_{n,F} \le
P_G(\Psi_{\eps_n})$.
Lemma ~\ref{lem:press cts} implies that we can  take the limit to get
$P_G(\Psi) = \lim_{n\to\infty} P_G(\Psi_{\eps_n}) = 0$.
\end{proof}

We next present a technical result, which when applied
to the settings of Theorems~\ref{thm:poly} and \ref{thm:CE},
shows that any measure with free energy close to our equilibrium states
lifts to a single inducing scheme, see Lemma~\ref{lem:uniform scheme for nat}.

Lemma~\ref{lem:compat entro} says that given $\eps>0$ there exists $\eta=\eta(\eps)$ and $\hat E=\hat E(\eps)$, a compact set bounded away from $\bd\D$, so that $h_\mu(f)>\eps$ for $\mu\in \M$ implies $\hat\mu(\hat E)>\eta$.
This implies that for a measure $\mu \in \M_+$,
in particular an equilibrium state $\mu_\psi$, we can
choose $X^0\in \P_n$ so that for the set $\hat X^0$ as in Theorem~\ref{thm:lyap lift} (or Lemma~\ref{lem:exten} if a first extendible return map is preferred)
$\hat \mu_\psi(\hat X^0 \cap \hat E) > 0$.
Next we add a finite collection of
cylinder sets $X^k \in \cup_{j\ge n}\P_j$, $k = 1, \dots, N$, so that
if we create the sets $\hat X^k \subset \pi^{-1}(X^k)$ in the
same way (i.e., as in
Theorem~\ref{thm:lyap lift} or as in Lemma~\ref{lem:exten}),
then  $\hat E \subset \left(\cup_{0\le k\le N}\hat X^k\right)$.
In this case we say that $\{\hat X^k\}_{0\le k\le N}$ satisfies
property $Cover(\eps)$.
The next proposition shows that there is a single inducing scheme that is compatible to every measure in $\M_+$ whose free energy is sufficiently close
to the pressure.

\begin{proposition}
Suppose that $\psi:I\to [-\infty, \infty)$ is a potential with $P(\psi)=0$
so that  $\psi(x) > -\infty$ on $I \sm \Crit$.
Suppose also that there exist $\eps_0,\, \eps>0$ \st
$h_{\mu'}(f)+\int\psi~d\mu' > -\eps_0$ implies $h_{\mu'}(f)>\eps$.
Let $\{\hat X^k\}_{0\le k\le N}$ satisfy $Cover(\eps)$ where $\mu_\psi$ is
compatible to $(X^0,F_0)$.
Suppose that the induced potentials $\Psi^k$
and inducing times $\tau^k$ corresponding to the inducing schemes
$(X^k, F_k)$ satisfy:
\begin{itemize}
\item[(a)] $\sum_nV_n(\Psi^k)<\infty$ for all $0\le k\le N$;
\vspace*{1mm}
\item[(b)] $\sum_i \tau_i^k e^{\sup\{\Psi^k(x)\, :\, x \in X_i^k\}} < \infty$
(i.e.,  condition (a) of Proposition~\ref{prop:all_indu} holds for
$\Psi^k$)  for all $0\le k\le N$.
\end{itemize}
Then there exists $\theta=\theta(\eps,\{\hat X^k\}_{0\le k\le N})>0$ so that $h_{\mu}(f)+\int\psi~d\mu>-\theta$ implies $\hat\mu(\hat X_0)>0$. \label{prop:unif scheme}
\end{proposition}

The idea here is that information on the equilibrium state for $(X^0, F_0, \Psi^0)$ allows us to show that measures with enough free energy must cover a large portion of the Hofbauer tower, in particular they are compatible to $(X^0,F_0)$.

\begin{proof}
Let $k \in \{ 1, \dots, N\}$ be arbitrary and assume that $\mu' \in \M_+$
is a measure such that $\hat\mu'(\hat X^k)>0$, but with $\hat\mu'(\hat X^0)=0$.

Here we will refer to the components of $\pi^{-1}(X^k_i)\cap \hat X^k$
as \emph{$1$-cylinders of $(\hat X^k, R_{\hat X^k})$},
the first return map to $\hat X^k$.
\begin{claim}
\begin{itemize}
\item[(i)] There is at least one $1$-cylinder mapping into $\hat X^0$
before returning to $\hat X^k$;
\item[(ii)] There is at least one $1$-cylinder which does not map to $\hat X^0$ before returning to $\hat X^k$.
\end{itemize}
Moreover, whether (i) or (ii) holds depends only on $\pi(\hat X^k_i)$,
and not on the domain that $\hat X^k_i$ belongs to.
\end{claim}
\begin{proof}  Property (i) follows by transitivity.
(A priori, sets $\hat X^k_i$ satisfying (i) may have
$\hat\mu'(\hat X^k_i) = 0$
or not; we will show that$\hat\mu'(\hat X^k_i) > 0$ for at least one such $\hat X^k_i$.)

For property (ii), suppose that for any first
return domain $\hat X_i^k\subset D\in \D$ there is
$0 \le s < r_{\hat X^k}(\hat X_i^k)$ \st
$\hat f^s(\hat X_i^k)\cap\hat X^0\neq \es$.  By the properties of
cylinders we must in fact have $\hat f^s(\hat X_i^k)\subset\hat X^0$.
This means that $\hat\mu'$-a.e. point enters $\hat X^0$ with positive frequency.  Ergodicity implies that $\hat\mu'(\hat X^0)>0$ which is a
contradiction. Hence (ii) holds.

Since $\hat X^k \in \cup_{j \ge n} \P_j$, if (i) holds for some
$1$-cylinder $\hat X^k_i$ of  $(\hat X^k, R_{\hat X^k})$,
say, then this whole cylinder maps into $\hat X^0$.
Moreover, by the proof of Lemma~\ref{lem:exten}, see \cite{BrCMP},
if $\hat y_1, \hat y_2 \in \hat X^k$ have $\pi(\hat y_1) = \pi(\hat y_2)$
and $\hat f^k(\hat y_1)\in \hat X^0$ then $\hat f^k(\hat y_2)\in \hat X^0$.
Consequently, for a $1$-cylinder $X_i^k$ of $(X^k,F_k)$ either every component
of $\pi^{-1}(X_i^k) \cap\hat X^k$ has property (i), or every component of
$\pi^{-1}(X_i^k) \cap\hat X^k$ has property (ii).
This concludes the proof of the first claim.
\end{proof}

Since, by the Gibbs property from Theorem~\ref{thm:BIP},
$\mu_\Psi$ gives all cylinders of
$(X^0,F_0)$ positive mass, the same must be true of the
$\hat\mu_\psi\circ\pi|_{\hat X^0}^{-1}$-measure of these cylinders.
Thus part (i) of the claim implies that $\hat\mu_\psi(\hat X^k)>0$
and hence $\mu_{\psi}$ is compatible to $(X^k, F_k)$.
By Case 3 of Proposition~\ref{prop:all_indu}, this also implies that
$P_G(\Psi^k)=0$.

Let $(X^k_{\flat}, F_k)$ denote the system minus the cylinders satisfying (i).
Let $P_G^{\flat}(\Psi^k)$ denote the Gurevich pressure of
$(X^k_{\flat}, F_k, \Psi^k)$, computed from $Z^{\flat}_n(\Psi^k)$,
which is defined in the natural way.
(Note that one consequence of part (ii) of the claim
is that $P^\flat_G(\Psi^k)>-\infty$.)

\begin{claim}
$P^\flat_G(\Psi^k)<P_G(\Psi^k)=0$.
\end{claim}

\begin{proof}
Let $\mathcal{Y}^k$ be the union of $1$-cylinders of $(X^k, F_k)$
whose representatives in $\hat X^k$ satisfy property (i).
We fix a $1$-cylinder $Y^k$ so that $Y^k \cap \mathcal{Y}^k=\es$,
i.e., its representatives in $\hat X^k$ satisfy (ii).
In each $\c^k_j\subset Y^k$ there exists a unique periodic point which
contributes to $Z_j(\Psi^k,  Y^k)$.
Thus noting that
$m_{\Psi^k}(\c^k_j) = \int_{\c^k_j} e^{-\Psi^k(x)} d\mu_{\Psi^k}$
and using the variation properties of $\Psi^k_j$, we derive
$$
e^{-V_j(\Psi^k)} \sum m_{\Psi^k}(\c^k_j) \le
Z_j(\Psi^k,  Y^k) \le e^{V_j(\Psi^k)}
\sum m_{\Psi^k}(\c^k_j)
$$
where the sum is taken over all $j$-cylinders $\c^k_j$ in $Y^k$.
Similarly
$$
e^{-V_j(\Psi^k)} \sideset{}{{}^\flat}\sum m_{\Psi^k}(\c^k_j) \le
Z^\flat_j(\Psi^k,  Y^k)
\le
e^{V_j(\Psi^k)} \sideset{}{{}^\flat}\sum m_{\Psi^k}(\c^k_j)
$$
where the sum $\sum^\flat$ is taken over all $j$-cylinders $\c^k_j$
in $Y^k$ so that $F_k^s(\c^k_j)\cap \mathcal{Y}^k =\es$ for $0\le s \le j-1$.

For every $\c^k_j$ in the sum $\sum^\flat m_{\Psi'}(\c^k_j)$ there exist
collection of $j+1$-cylinders $\c^k_{j+1}$ so that
$F_k^{j}(\cup \c^k_{j+1}) = \mathcal{Y}^k$.
Since $m_{\Psi^k}$ is conformal and $\Psi^k$ has summable variations,
we have
$$
\frac{m_{\Psi^k}(\cup \c^k_{j+1})}{m_{\Psi^k}(\c^k_j)} \ge
\frac1K \left(\frac{m_{\Psi^k}(\mathcal{Y}^k)}{m_{\Psi^k}(X^k)}\right)
$$
where $K=e^{\sum_jV_j(\Psi^k)}$.
Hence, since $m_{\Psi^k}(X^k)=1$,
\begin{eqnarray*}
\sideset{}{{}^\flat}\sum m_{\Psi^k}(\cup \c^k_{j+1})
&=& \sideset{}{{}^\flat} \sum (m_{\Psi^k}(\c^k_j) - m_{\Psi^k}(\cup \c^k_{j+1}))\\
&\le&
\left(1-\frac{m_{\Psi^k}(\mathcal{Y}^k)}{K}\right)
\sideset{}{^\flat}\sum m_{\Psi^k}(\c^k_j).
\end{eqnarray*}
Letting $\xi:=\frac{\mu_{\Psi^k}(\mathcal{Y}^k)}{K}$ we have
$$
Z^\flat_{j+1}(\Psi^k,  Y^k) \le e^{V_{j+1}(\Psi^k)}
\sideset{}{{}^\flat}\sum\mu_{\Psi^k}(\c^k_j)
\le e^{V_{j+1}(\Psi^k)+V_j(\Psi^k)}(1-\xi)Z^\flat_j(\Psi^k,  Y^k).
$$
Therefore $Z^\flat_{n}(\Psi^k,  Y^k) \le e^{2\sum_jV_j(\Psi^k)} (1-\xi)^n
Z^\flat_{1}(\Psi^k,  Y^k)$.
Since Lemma~\ref{lem:sum var} implies
$\sum_jV_j(\Psi^k)<\infty$, we have $P^\flat_{G}(\Psi^k)<\log(1-\xi)<0$,
as required.
This completes the proof of the second claim.
\end{proof}

Now take $\theta_k > 0$ 
so that $P^\flat_{G}(\Psi^k+\theta_k\tau^k)\le 0$.
If the measure $\mu'$ from the beginning of the proof
satisfies $h_{\mu'}(f) + \int \psi d\mu' > -\theta_k$, then
$h_{\mu'}(f) + \int (\psi + \theta_k) d\mu' > 0$, so
Lemma~\ref{lem:Abra} implies that the corresponding induced measure
$\mu'_{F_k}$ has
$h_{\mu'_{F_k}}(F_k)+\int(\Psi^k+\theta_k\tau^k)~d\mu'_{F_k}>0$.
From the Variational Principle for the system
$(X^k_\flat, F_k, \Psi^k+\theta_k\tau^k)$
we see that $\mu'_{F_k}$ cannot be supported on type (ii) $1$-cylinders of
$(X^k, F_k)$ only. Hence $\hat\mu'(\hat X^0) > 0$.

Finally take $\theta := \min\{ \eps_0, \theta_1, \dots, \theta_N\}$
and let $\mu$ be such that
$h_{\mu}(f) + \int \psi d\mu > -\theta$. Since $\theta \le \eps_0$, we have
$h_{\mu}(f) > \eps$ by assumption,
and therefore $\mu$ is compatible to $(X^k, F_k)$ for some
$k \in \{0, 1, \dots, N\}$.
By the choice of $\theta$ and the
argument of the previous paragraph, it follows that $\hat\mu(\hat X^0) > 0$ as required.
\end{proof}

\section{Proofs of Theorem~\ref{thm:poly} and \ref{thm:CE}}
\label{sec:poly}

Let $\phi = \phi_t = -t \log |Df|$, and $\Phi$
be the corresponding
induced potential.  Przytycki \cite{Prz} proves that a measure $\mu\in \M$ is either supported on an attracting periodic orbit or $0\le \int\log|Df|~d\mu<\infty$.  So when we apply Lemma~\ref{lem:Abra} to this potential, we will get finite integrals for both the measure on $I$ and for the measure on the inducing scheme with the induced potential.

\begin{lemma} Assume that $f$ has negative Schwarzian derivative.
For inducing schemes obtained in Section~\ref{sec:induce},
the induced potential has summable variations. \label{lem:sum var}
\end{lemma}

\begin{proof}
In general, $\phi$ has unbounded variations. However, we note that
inducing schemes as in Theorem~\ref{thm:lyap lift} and Lemma~\ref{lem:exten} are maps  $F:\bigcup_j X_j \to X$ with uniform Koebe space $\delta$.
Since $\phi$ is in general unbounded, it will not have bounded
variations, but we only need to check that the induced potential
$\Phi$ has bounded variations. By the Koebe Lemma,
$\frac{|DF(y)|}{|DF(x)|} <\frac{1+2\delta}{\delta^2}+1$. Therefore
\begin{align*} |\Phi(x)-\Phi(y)| & = |t| \ \bigg|-\log|DF(x)|+
\log|DF(y)|\bigg| =
|t| \ \left|\log\left(\frac{|DF(y)|}{|DF(x)|}\right)\right|\\
& \le |t| \ \log\left(1+\frac{1+2\delta}{\delta^2} \right)< |t|
\left(\frac{1+2\delta}{\delta^2}\right).
\end{align*}
By standard arguments, for any $\gamma>1$ there exists $N=N(\gamma)$
such that we have $\inf_{x\in X}|DF^N(x)|>\gamma$ (here we use the negative
Schwarzian assumption; alternatively a $C^3$ assumption and the
absence of neutral periodic cycles would suffice).  Moreover, $F^N$
satisfies the above distortion estimates.  Let $\gamma>\frac1\delta$
and let $G:\bigcup_j Y_j\to X$ be given by $G:=F^{N}$ for
$N=N(\gamma)$.  Clearly, proving the lemma for $\Phi_N$ is sufficient.

We have that $X$ is a $\gamma\delta$-scaled neighbourhood of
$Y_j$ for any $j$.  Using the Koebe Lemma again for $x,y$ in the same
connected component of $G^{-1}(Y_j)$, we have
\[
|\Phi_N(x)-\Phi_N(y)|< |t| \ \left(\frac{1+2\gamma\delta}
{(\gamma\delta)^2}\right).
\]
Repeating this argument for $x,y$ in the same connected component
of $G^{-n}(Y_j)$ that
\[
|\Phi_N(x)-\Phi_N(y)|< |t| \
\left(\frac{1+2\gamma^n\delta}{(\gamma^n\delta)^2}\right) =
|t| O(\gamma^{-n}).
\]
Thus $\Phi_N$, and hence $\Phi$, has summable variations.
\end{proof}

The proofs of Theorems~\ref{thm:poly} and \ref{thm:CE} have
roughly the same structure. We start with the Collet-Eckmann
case, leaving the additional details for the summable case to the
end of the section. For use in both proofs, we define
\[
Z_0(\Phi):=\sum_{F(x)=x} e^{\Phi(x)}.
\label{page:Z_0}
\]
As stated in the proof of
Proposition~\ref{prop:all_indu}, we have $Z_n(\Phi)=O(\|\L_\Phi 1
\|_\infty^n)$. In this case, bounded distortion gives $\|\L_\Phi 1
\|_\infty \asymp_{dis} Z_0(\Phi)$. Thus $Z_n(\Phi) =O([Z_0(\Phi)]^n)$.

We are now ready to prove Theorem~\ref{thm:CE}, although we
postpone the proof that $t\mapsto P(\phi_t)$ is analytic to the
end of Section~\ref{sec:tails}.

\begin{proof}[Proof of the first part of Theorem~\ref{thm:CE}]
We choose $X$ as in Corollary~\ref{cor:uniform X} and apply the method of
Lemma~\ref{lem:exten} to get an extendible inducing scheme $(X,F)$.

Fixing $t$, we define $\psi_S = \phi_t-S$, and let $\Psi_S$ be the
induced potential.  The natural candidate for $S$ is $P(\phi_t)$,
but we will want to consider a more general value for this shift
in the potential in order for (c) of
Proposition~\ref{prop:all_indu} to hold.

We continue by showing that the induced system has bounded Gurevich
pressure and (a) and (c) of Proposition~\ref{prop:all_indu} hold.
As above, $Z_n(\Phi) =O(Z_0^n(\Phi))$. Therefore it suffices to
show that $Z_0(\Phi_S)<\infty$ to conclude that $P_G(\Psi_S)<\infty$.

We wish to count the number of domains $X_i$ with $\tau_i = n$.  The number of \emph{laps} of a piecewise continuous function $g$ is the number of maximal intervals on which
$g$ is monotone.  We denote this number by $\laps(g)$.  By \cite{MSz}, one characterisation of the topological entropy is
$\htop(f) := \lim_{n\to \infty} \frac1n \log
\laps(f^n)$.  Therefore, for all $\eps>0$ there exists $C_\eps>0$ \st
\[
 \#\{ \tau_i = n \} \le \laps(f^n) \le C_\eps e^{n (\htop(f)+\eps)}
\]
for each $n$, where $\htop(f)$ denotes the topological entropy of $f$. Since $f$ is Collet-Eckmann, the tail behaviour of the inducing scheme is exponential.  This was shown for certain inducing schemes in \cite{BLS}.  We show in the proof of Proposition~\ref{prop:summable} that the results on the inducing schemes of \cite{BLS} hold for the inducing schemes of Lemma~\ref{lem:exten}.  We also show there how \cite{BRSS} allows us to strengthen
the results of \cite{BLS} to apply to maps with different critical orders, see Lemma~\ref{lem:BCC} below.

For $t\le 1$ we get
\begin{align*}
Z_0(\Psi_S) &:= \sum_{F(x) = x} e^{\Psi_S(x)} = \sum_{i, x = F(x)
\in X_i} e^{\Phi_t(x) - \tau_i(x)S }\\
&
\asymp_{dis}  \sum_i |X_i|^{t} e^{-\tau_i(x)S }
= \sum_n \sum_{\tau_i = n} |X_i|^{t} e^{-nS}  && \mbox{by the Koebe Lemma} \\
&\le \sum_n \left( \sum_{\tau_i = n} |X_i| \right)^t e^{-nS}
\left(\#\{ \tau_i = n \}\right)^{1-t}
&& \mbox{by the H\"older inequality} \\
&\le C_\eps\sum_n e^{-\alpha n t} e^{-nS} e^{n (\htop(f)+\eps)(1-t)} < \infty
&& \mbox{using tail behaviour}
\end{align*}
provided $t$ is sufficiently close to 1 and $S > \htop(f)(1-t) -
\alpha t$. A similar estimate gives
\begin{equation}
\sum_i \tau_i e^{\Psi_S(x)} \asymp_{dis} \sum_i \tau_i |X_i|^t
e^{-\tau_i S} < \infty. \label{eq:int}
\end{equation}

For $t\ge 1$
\begin{align*}
Z_0(\Psi_S) &\asymp_{dis}  \sum_n \sum_{\tau_i = n} |X_i|^{t} e^{-nS}
\le  \sum_n e^{-nS} \left( \sum_{\tau_i = n} |X_i| \right)^t
\\
&\le \sum_n e^{-\alpha n t} e^{-nS} < \infty,
\end{align*}
provided $S > -\alpha t$. Similarly we can show
\[
\sum_i \tau_i e^{\Psi_S(x)} \asymp_{dis} \sum_i \tau_i |X_i|^t
e^{-\tau_i S} < \infty,
\]
provided $S > -\alpha t$.  When $t$ is sufficiently close to $1$,
$P(\phi_t)$ is close to 0, and thus if $S$ is close to $P(\phi_t)$ then the above sums are bounded.

Observe that the above estimates prove that condition (a) of
Proposition~\ref{prop:all_indu} holds.  For part (c) of that
proposition, the estimates above prove that
$P(\Psi_{P(\phi_t)+\eps})<\infty$ for $\eps<0$ close to $0$.
Therefore, Corollary~\ref{cor:uniform X} shows that (c) is be satisfied.
Therefore this inducing scheme gives rise to an equilibrium state $\mu_\phi=\mu_\psi$.  Moreover, from the proof of Proposition~\ref{prop:all_indu}, $P_G(\Psi)=0$.

It remains to show the uniqueness of the equilibrium state in
$\M_+$, since up to this point we only know that $\mu_\phi$ is the
unique equilibrium state whose lift to the Hofbauer tower gives $\hat X$ positive mass.  This follows from the next lemma.

\begin{lemma}
If $\mu_\phi$ is an equilibrium state, as above, compatible to an inducing scheme $(X,F)$ then it is also is compatible to any other inducing scheme $(X',F')$ provided $\hat X'\cap\E\neq \es$.  Here we assume that the inducing schemes are either both as in Theorem~\ref{thm:lyap lift} or both as in Lemma~\ref{lem:exten}. \label{lem:eq sees all}
\end{lemma}

\begin{proof}  We will assume that the inducing schemes here are all as in Lemma~\ref{lem:exten}, since this is the more difficult case. Let $(\hat X, \hat F)$ be the inducing scheme used above.  The proof follows if we can show that $\hat\mu_\phi(\hat X')>0$.

Transitivity of $(\E,\hat f)$ implies that there exists $n\ge 0$ so that $\hat f^{-n}(\hat X')\cap\hat X$ contains an open set.  As in Proposition~\ref{prop:unif scheme}, since $\mu_\Psi$ gives positive mass to cylinders, this implies that there exists $\hat U\subset\hat X$ so that $\hat\mu_\phi(\hat U)>0$
and $\hat f^n(\hat U)\subset \hat X'$.
Hence,
$$
\hat\mu_\phi(\hat X') \ge \hat\mu_\phi(\hat f^n(\hat U))\ge \hat\mu_\phi(\hat U)>0.
$$
Therefore, $\mu_\phi$ is compatible to $(X',F')$.
\end{proof}

Suppose that $\mu\in \M_+$ is an equilibrium state.  By the ideas of
Lemma~\ref{lem:exten} there must exist a first extendible inducing scheme
$(X',F',\Psi')$ which is compatible to $\mu$ and which corresponds
to a first return map to a set $\hat X'$  on the Hofbauer tower.
Lemma~\ref{lem:eq sees all} implies that $\mu_\phi$ is compatible to $(X',F')$ and hence $\mu=\mu_\phi$ by the uniqueness of equilibrium states on an inducing scheme.
\end{proof}

To do the summable case, we adapt techniques from \cite{BLS}.
In that paper, the Bounded Backward Contraction is used for arbitrary
neighbourhoods of the critical set, which at the time
was only known to hold when all critical orders $\ell_c$ are the same.
Using results from \cite{BRSS}, and specifying the neighbourhoods $U$,
we can improve this in the following lemma.

\begin{lemma}\label{lem:BCC}
Let $f \in \H$ be a multimodal map with negative Schwarzian derivative
such that $\lim_{n \to \infty} |Df^n(f(c))| = \infty$ for each $c \in \Crit$.
Then for any $\eps > 0$ and $\lambda > 1$, we can find critical neighbourhoods
$U := f^{-1}(B_\eps(f(\Crit)))$
that are \emph{$\lambda$-nice} in the sense that
\begin{itemize}
\item $f^n(\partial U) \cap U = \emptyset$ for all $n \ge 0$;
\item if $V \subset U$ is the domain of the first return map to $U$,
then the interval $V'$ concentric to $V$ and of length $(1+2\lambda)|V|$
is contained in $U$.
\end{itemize}
Moreover, there exists $b > 0$ such that
\begin{equation}\label{eq:BCC}
|Df^r(x)| \ge b \text{ for all } x \in I \text{ and }
r = \min\{ n \ge 0 : f^n(x) \in U \},
\end{equation}
where the $\lambda$-nice critical neighbourhood $U$ can be chosen
arbitrarily small.
\end{lemma}

\begin{proof}
The first part follows immediately from \cite{BRSS} which considers
$C^3$ non-flat
multimodal maps. Our assumption that $f$ is $C^2$ with
 negative Schwarzian derivative actually gives a slightly stronger version
of the Koebe distortion theorem, and hence is sufficient to claim the results
from \cite{BRSS}.
Lemma 3 in \cite{BRSS} shows the existence of $\lambda$-nice
neighbourhoods $U$ of $\Crit$. Denote the connected components of $U$
by $U^c$, $c \in \Crit$.
If $r = r(x) \ge 0$ is the first entrance time of $x$ to $U$,
then the niceness
of $U$ guarantees that there exists an interval $J_x$ so that $f^r$ maps
$J$ diffeomorphically onto $U^c$ for some $c \in \Crit$.
If $f^r(x)$ belongs to first return domain $V$, then there is $J_V \subset J$
such that $f^r:J_V \to V$ is monotone with distortion bound depending only on
$\lambda$.
A special case of this is when $V := \tilde U^c$ is the central return
domain in $U^c$.
Let $\tilde U = \cup_{c \in \sCrit} \tilde U^c$. In this case, the first entrance
time $\tilde r \ge 0$ of any
$x$ into $\tilde U$ corresponds to a diffeomorphic branch
$f^{\tilde r}:\tilde J \to \tilde U^c$ with distortion bound depending only
on $\lambda$.

\begin{remark}
Note that $U \subset f^{-1}(B_\eps(f(\Crit))$, where $\eps$ can be
taken arbitrarily small. As a result, the components $U^c$
need not have comparable sizes for all $c \in \Crit$,
but scale as $\eps^{1/\ell_c}$.
A similar difference in size is true for the components of $\tilde U$,
and this is
a major difference with the critical neighbourhoods as used in \cite{BLS}.
If all components of $\tilde U$ have the same size,
then \eqref{eq:BCC} can fail.
\end{remark}

To prove \eqref{eq:BCC}, fix
a $\lambda$-nice critical neighbourhood $U_0$,
and let $U_1 := \tilde U_0$ be the union of its central
return domains. This set is $\lambda$-nice again.
There exists $b = b(U_1) > 0$ such that for every $x \in I$,
$|Df^{r_1}(x)| \ge b$ for $r_1 = \min\{ n \ge 0 : f^n(x) \in U_1 \}$.
Continue to construct $\lambda$-nice neighbourhoods $U_i = \tilde U_{i-1}$ as
the union of the central return domains of the previous stage.
These set shrink at least exponentially in $i$, so
 we obtain a $\lambda$-nice neighbourhood
$U = U_p$ as small as we want.

Now let $r_1 \le r_2 \le \dots \le r_p = r$ be the return times of
$x$ to $U_1 \supset U_2 \supset \dots \supset U_p$.
There is a neighbourhood $J \owns x$ such $f^r$ maps $J$ diffeomorphically
onto a component of $U$.
The maps $f^{r_{i+1}-r_i}|_{f^{r_i}(J)}$ are composition of monotone branches
of the first return map to $U_i$. If $\lambda$ is sufficiently large, then
these branches are expanding, uniformly in $x$.
Hence $|Df^r(x)| \ge |Df^{r_1}(x)| \ge b$.
\end{proof}

\begin{proposition} \label{prop:summable}
Suppose that $f$ is a multimodal map
satisfying \eqref{eq:polynomial}. Then on every sufficiently small cylinder set $X$ there is a first extendible return inducing scheme $(X, F,\tau)$ and $t_1 \in [t_0, 1]$ \st for all $t \in (t_1,1]$:
and all potential shifts $S \ge 0$:
\[
Z_0(\Psi_S) := \sum_{F(x) = x} e^{\Psi_S(x)} < \infty,
\]
where $\Psi_S$ is the induced potential of the shifted potential
$\psi_S := \phi_t - S$.  Furthermore for the equilibrium state $\mu_{\Psi_{P(\phi_t)}}$, $\mu_{\Psi_{P(\phi_t)}}\{\tau=n\}$ decays exponentially for $t\in (t_1,1)$, and polynomially for $t=1$.
\end{proposition}

\begin{proof}
For the case $t=1$, if the critical points all have the same order then \cite{BLS} gives an inducing scheme with polynomial tails (this is also sufficient to show $Z_0(\Psi_S)<\infty$ for all $S\ge 0$).  Below we show that inducing schemes from Lemma~\ref{lem:exten} fit into the framework of \cite{BLS}.  We also show that by  Lemma~\ref{lem:BCC}, the machinery of \cite{BLS} can also be applied to maps with critical points with different critical orders, by Lemma~\ref{lem:BCC}.  We focus on the details of the case $t<1$,  showing that these systems have exponential tails.  The proof that our inducing schemes give equilibrium states with polynomial tails for $t=1$ is left to the reader.
From here onwards, we restrict our proof to the case $t<1$.

Fix a single cylinder set $X \in \P_n$
and $\delta \in (0,\frac12)$ so small that a $\delta$-scaled
neighbourhood of $X$ is contained in $\pi(D)$ for at least one domain $D$ of the closed primitive subgraph $\E$ (cf. Lemma~\ref{lem:trans_subgraph})
of the Hofbauer tower. The inducing scheme will be the first extendible return to $X$ in the sense of Lemma~\ref{lem:exten}: namely, for each $X_i$, there is a neighbourhood $X_i'$ \st $f^{\tau_i}$ maps $X'_i$ diffeomorphically onto a $\delta$-scaled neighbourhood $X$.  Let $\hat X \subset \pi^{-1}(X)$ be \st the inducing scheme corresponds to the first return map to $\hat X$.
Since $X$ is a cylinder set, $\hat X$ is {\em nice} in the sense that for $n\ge 1$, $\hat f^n(\hat x)$ never intersects the interior of $\hat X$ for each $\hat x \in \partial \hat X$.  There is a dense orbit $\orb(\hat y)$ in $\E$, and for each visit $\hat y' \in  \orb(\hat y) \cap \hat X$, there is a
neighbourhood $\hat X_i \owns \hat y'$ \st $\hat f^{\tau_i}:\hat X_i \to \hat X$ is extendible to a $\delta$-scaled neighbourhood of a component of $\hat X$.  Therefore, the union $\cup_i X_i$ (and hence $X^{\infty}$) is dense in $X$, and the niceness of $\hat X$ guarantees that the sets $X_i$ are pairwise
disjoint.

Note that \eqref{eq:polynomial} implies that
\begin{equation}\label{eq:summable}
\sum_n \left( \gamma_n^{\ell_c-1} |Df^n(c_1)|\right) ^{-t_0/\ell_c} <
\infty,
\end{equation}
for every $c \in \Crit$, some $t_0<1$ and summable sequence
$\{ \gamma_n\}_{n \in \N}$
with $\gamma_n \in (0,\delta|X|)$. Throughout we can take
$\gamma_n = \frac{\delta|X|}{n \log^2(n+10)}$.

We use ideas and results of \cite{BLS} extensively. To start with,
given a neighbourhood $U$ of $\Crit$
as in Lemma~\ref{lem:BCC} (so that \eqref{eq:BCC} holds),
we can assign to any $x \in
I$ a sequence of {\em binding periods} along which the orbit of
$x$ shadows a critical orbit, followed by {\em free period} during
which the orbit of $x$ remains outside $U$. During the binding
period, derivative growth is comparable to derivative growth of
the critical orbit. The precise definition of binding period of $x
\in U$ is:
\[
p(x) = \min\{ k \ge 1 : |f^k(x) - f^k(c)| \ge \gamma_k |f^k(c) -
\Crit| \},
\]
where $c$ is the critical point closest to $x$. At the end of the
binding period, derivatives have recovered from the small
derivative incurred close to $c$. Indeed, Lemma 2.5 of \cite{BLS}
states that there is $C_0 > 0$, independent of $U$, \st
\[
F'_p(x) := \inf\{ |Df^p(x)| : x \in U, p(x) = p \} \ge C_0 \left(
\gamma_p^{\ell_c-1} |Df^p(f(c))| \right)^{1/\ell_c}.
\]
where $c$ is the critical point closest to $x$. If $U$ is a small
neighbourhood, then $p(x)$ is big. Hence we can take $U$ so small
that the minimal binding period $p_U := \min\{ p(x) : x \in U\}$
is so large that Equation (5) in \cite{BLS} holds:\footnote{Here we take into account the typo in Equation (5) of \cite{BLS} where the $-$ in the exponent is missing.}
\begin{equation}\label{eq:xxx}
\max_{c \in \sCrit} \sum_{s \le n} \ \sum_{
\substack{(p_1, \dots, p_s)\\
\sum_i p_i \le n \\
p_i \ge p_U} } \ \prod_{p_i} \zeta \left(\ \gamma_{p_i}^{\ell_c-1}
|Df^{p_i}(f(c))|\ \right)^{-1/\ell_c} \le 1.
\end{equation}
Here $\zeta = 4C_4\#\Crit$ (see later in the proof)
is a fixed number involving a Koebe constant and a
constant emerging from the Bounded Backward Contraction Condition
\eqref{eq:BCC}, see Lemma~\ref{lem:BCC}.
The constant $\zeta$ is independent of $U$.
\\[0.2cm]
During the {\em free period}, derivatives grow exponentially
(Ma\~n\'e's Theorem, see \cite[Theorem III.5.1.]{MSbook}),
because there exist $C_1 > 0$ and $\lambda_1 > 1$, depending only
on $f$ and $U$, \st
\begin{equation}\label{eq:outsideX}
|Df^k(x)| \ge C_1 \lambda_1^k \quad \mbox{ if } \quad f^i(x)
\notin U \mbox{ for } 0 \le i < k.
\end{equation}
Now fix a neighbourhood $U$ of $\Crit$ as in Lemma~\ref{lem:BCC} with $\partial U \subset
\cup_n f^{-n}(\Crit)$ and so small that estimate \eqref{eq:xxx}
holds.  In fact, parallel to \eqref{eq:outsideX}, one can derive
sets that avoid $U$ for a long time are exponentially small: there are $C_a > 0$ and $\lambda_2 > 1$ \st
\begin{equation}\label{eq:outsideX2}
|f^n(A)| \le C_a \lambda_2^{-k} \quad \mbox{ if } \quad f^i(A)
\cap U = \emptyset \mbox{ for } 0 \le i < k.
\end{equation}

Since $\partial U$ consists of precritical points, and each $X_i$
is mapped monotonically onto $X$, there is $\kappa$ \st $f^j(X_i)
\cap \partial U \neq \emptyset$ implies $j \ge \tau_i - \kappa$.
Given $X_i$ and $j < \tau_i - \kappa$, $f^j(X_i)$ will either
be contained in or disjoint from $U$.  Thus we can define  $\nu_j(X_i)$ to be the time at which the $j$-th binding period starts and the binding periods itself as $p_j(X_i) = \min\{ p_j(x) : x \in X_i\}$.  Since $f^{\tau_i - n}$ maps $f^{n}(X_i)$ to $X$ in an extendible way for each $n \le \tau_i$, the distortion of $f^{\tau_i - n}|_{f^n(X_i)}$ is bounded uniformly in $i$ and $n$. We will write $\nu_j = \nu_j(X_i)$ and $p_k=p_k(X_i)$ if it is clear from the
context which $X_i$ is meant. Note that the inducing time $\tau_i$
of $X_i$ cannot be inside a binding period, because during the
binding period, $X_i$ shadows some critical value $f^k(c)$
$\gamma_k$-closely, and $\gamma_k < \delta|X|$ for every $k$.

In the terminology of \cite{BLS}, every return time is a {\em deep return}, and there are no {\em shallow returns}.  Let $\tau'_i$ be the time that the final binding period ends, so $\tau'_i = \nu_s + p_s \le \tau_i$ if $X_i$ has $s$ binding periods.

To estimate $Z_0(\Psi_S)$, we first group together domains $X_i$
into a  `cluster' if they have the same binding periods $p_1, \dots,
p_s$ up to their common time $\tau'_i$ and $f^j(\conv \tilde
A) \cap \Crit = \emptyset$ for $j \le \tau_i$, where $\conv \tilde A$
is the convex hull of the cluster.  We have by the H\"older inequality
\begin{align*}
Z_0(\Psi_S) &\asymp_{dis} \sum_i |X_i|^t e^{-\tau_i S} = \sum_n e^{-nS}
\sum_{n' \le n} \sum_{ \stackrel{\mbox{ \tiny cluster }\tilde
A}{\tau(\tilde A) = n, \tau'(\tilde A)= n'} }
\sum_{X_i \subset \tilde A} |X_i|^t  \\[0.05cm]
&\le \sum_n e^{-nS} \sum_{n' \le n} \sum_{ \stackrel{\mbox{ \tiny
cluster } \tilde A}{\tau(\tilde A) = n, \tau'(\tilde A)=
n'}}
(\#\{ i : X_i \mbox{ belongs to } \tilde A\})^{1-t}
\left( \sum_{X_i \subset \tilde A}
|X_i| \right)^t   \\[0.05cm]
&\le \sum_n e^{-nS} \sum_{n' \le n}
e^{(\htop(f)+\eps)(n-n')(1-t)}
\sum_{ \stackrel{\mbox{ \tiny cluster }\tilde A} {\tau(\tilde A) =
n,\ \tau'(\tilde A)= n'}} |{\rm conv}\tilde A|^t,
\end{align*}
where the cardinality $\#\{ i : X_i \mbox{ belongs to } \tilde A\}$ is
estimated by $e^{(\htop(f)+\eps)(n-n')}$ for some small $\eps =\eps(t)>
0$, because the cluster $\tilde A$ has $n-n'$ iterates left to the
inducing time.

To estimate $\sum_{\tau(\tilde A) = n, \tau'(\tilde A)= n'}
|\tilde A|^t$, we distinguish two classes of clusters depending on
the amount of free time in the first $\tau'$ iterates. For
$\eta > 0$ to be fixed later, and for given $n$ and $n'$, let
\[
\hat{\mathcal P}'_{n,n'}=\left\{ \tilde A \ : \ \tau'(\tilde A) =
n',\tau(\tilde A) = n, \, \sum_{i=1}^s p_i \le \eta n \right\}
\]
and
\[
\hat{\mathcal P}''_{n,n'} = \left\{ \tilde A \ : \ \tau'(\tilde A) =
n',\tau(\tilde A) = n, \, \sum_{i=1}^s p_i > \eta n \right\}.
 \]
The estimates for $\hat{\mathcal P}'_{n,n'}$ and $\hat{\mathcal
P}''_{n,n'}$ will use Lemmas 3.5 and 3.6 of \cite{BLS} respectively.
Indeed, Lemma 3.5 of \cite{BLS} gives some $\eta$ (fixing the
definition of $\hat{\mathcal P}'_{n,n'}$) and $\lambda_3 > 1$ depending on $\lambda_1$ and $\eta$ \st
\begin{equation}\label{eq:P'}
\sum_{\tilde A \in \hat{\mathcal P}'_{n,n'}} |\tilde A|^t \le
\lambda_3^{-\frac{1}{2}n' t} \sup_{\tilde A \in \hat{\mathcal
P}'_{n,n'}} |f^{n'}(\tilde A)|^t \le C_1^{-t} \lambda_3^{-\frac{1}{2}n'
t} \lambda_1^{-(n-n')t},
\end{equation}
where the last inequality follows by \eqref{eq:outsideX} because $f^{n'}(\tilde A)$ is disjoint from $U$ for the remaining $n-n'$ iterates.

Continuing with this $\eta$, define $d_n(c) := \min_{i < n}
(\gamma_i/|Df^i(f(c))|)^{1/\ell_c}|f^i(c) - \Crit| \le 1$ (formula (2) in \cite{BLS}) and let (following \cite[page 635]{BLS})
\[
\hat d_{n,j}(c) = d_i(c)  \text{ for }  i = \max\left\{ \left\lceil
\frac{\eta n}{2j^2} \right\rceil , 1\right\}.
\]
Then an adaptation of Lemma 3.6 of \cite{BLS} gives a constant
$C_2 > 0$ \st
\begin{equation}\label{eq:mainlemma}
\sum_{\tilde A \in \hat{\mathcal P}''_{n,n'}} |\tilde A|^t \le C_1^{-t}
\lambda_1^{-(n-n')t}  C_2 \sum_{s=1}^{n'} 2^{-j} \left( \max_{c
\in \sCrit}\hat d_{n',j} \right)^t.
\end{equation}
Indeed, select the longest binding period among $(p_1, \dots,
p_s)$ of the cluster, and call it $p_j$. Note that $p_j > \eta
n/(2j^2)$, because otherwise $\sum_{k=1}^s p_k <  \eta n$,
contradicting the definition of $\hat{\mathcal P}''_{n,n'}$. The
interval $[x,y] := f^{\nu_j}(\conv \tilde A)$ satisfies
\[
|x-y| \le C_3 \ \max_{p \ge \eta n/2j^2} d_p(c) \cdot
|f^{\nu_j+p_j}(\conv \tilde A)| = C_3 \ \hat d_{n',j}(c) \cdot
|f^{\nu_j+p_j}(\conv \tilde A)|,
\]
where $C_3$ is a uniform distortion constant. Write $\tilde A = \tilde A_{p_1, \dots, p_{j}}$ to indicate that $p_j$ is the longest binding period of $\tilde A$. By Lemma 3.2 of \cite{BLS}, and recalling that all returns are {\em deep}, we can find $C_4$ \st
\[
|\tilde A_{p_1, \dots, p_j}|
\le C_4^{j-1} \ |f^{\nu_{j-1}+p_{j-1}}(\conv \tilde A_{p_1, \dots, p_j})| \
\prod_{k=1}^{j-1} \frac{1}{F'_{p_k}}.
\]
Following the proof of Lemma 3.6  of \cite{BLS}, we obtain
\begin{align*}
\sum_{ \stackrel{\mbox{ \tiny cluster }\tilde A}{\tau(\tilde A) =
n, \tau'(\tilde A)= n'}} \!\!\! |\tilde A|^t &\le \sum_{j=1}^{n'}
\sum_{(p_1, \dots, p_j)}
|\tilde A_{p_1, \dots ,p_j}|^t \\
&\le \sum_{j=1}^{n'} \left( C_3 \max_{c \in \sCrit} \hat
d_{n',j}(c)\right)^t  \\
& \hspace{5mm} \times \sum_{(p_1, \dots, p_j)} (2\#\Crit)^j
\left( C_4^{j-1} \prod_{k=1}^{j-1} \frac{1}{F'_{p_k}} \right)^t
|f^{\nu_j+p_j}(\conv \tilde A_{p_1, \dots ,p_j})|^t,
\end{align*}
where the $(2\#\Crit)^j$ accounts for the different sides of
critical points that have intervals with the same binding period.
Using \eqref{eq:xxx} with $\zeta = 4C_4\#\Crit$, we can
estimate this by
\[
\sum_{j=1}^{n'} \left( C_3 \max_{c \in \sCrit} \hat
d_{n',j}(c)\right)^t \cdot 2^{-j} \cdot |f^{\nu_j+p_j}(\tilde
\conv A_{p_1, \dots ,p_j})|^t.
\]
The maps $f^{\nu_j+p_j}|_{\mbox{\tiny conv }\tilde A_{p_1, \dots ,p_j}}$ and $f^{n'-(\nu_j+p_j)}|_{f^{\nu_j+p_j}(\mbox{\tiny conv }\tilde A_{p_1, \dots ,p_j})}$ have bounded distortion. Each set $f^{n'}(\tilde A_{p_1, \dots ,p_j})$ is disjoint from $U$ for the
remaining $n-n'$ iterates, so \eqref{eq:outsideX2} gives
$|f^{n'}(X_i)| \le C_1^{-1} \lambda_2^{-(n-n')}$. Therefore
\[
\sum_{ \stackrel{\mbox{ \tiny cluster }\tilde A} {\tau(\tilde A) =
n, \tau'(\tilde A)= n'}} \!\!\! |\tilde A|^t \ \le C_1^{-t}
\lambda_2^{-(n-n')t} C_2 \sum_{j=1}^{n'} 2^{-j} \left(\max_{c \in
\sCrit} \hat d_{n',j}(c)\right)^t ,
\]
for $C_2 = (C_3 C_4)^t$. This proves \eqref{eq:mainlemma}.

Now we obtain (using \eqref{eq:mainlemma} and \eqref{eq:P'})
\begin{align*}
Z_0(\Psi_S) &\le \sum_n e^{-nS} \sum_{n' \le n}
e^{(\htop(f)+\eps)(n-n')(1-t)} \left( \sum_{\tilde A \in \hat{\mathcal
P}'_{n,n'}} |\tilde A|^t + \sum_{\tilde A \in \hat{\mathcal P}''_{n,n'}}
|\tilde A|^t
\right) \\
& \hspace{-10mm} \le \sum_n e^{-nS}  \sum_{n' \le n}
e^{(\htop(f)+\eps)(n-n')(1-t)}\lambda_2^{-(n-n')t}
 \left( \lambda_3^{-\frac12 n't} +
\sum_{j=1}^{n'} 2^{-j} \left(\max_{c \in \sCrit} \hat
d_{n',j}(c)\right)^t \right),
\end{align*}
which is finite, provided $t$ is sufficiently close to $1$. The
proof that $\int\tau~d\mu_\Psi<\infty$ amounts to showing that
$n e^{-nS} \sum_{n' \le n} \sum_{\tau_i = n , \tau'_i = n'} |X_i|^t$
is summable in $n$, cf. \eqref{eq:int}.  If $t < 1$, then $S = P(\phi) > 0$ by \eqref{eq:pressure}, so for $t$ sufficiently close to 1, the exponential factor $e^{-nS}$ dominates $n$ and summability follows.  This also implies the required exponential tails property for $(X,F,\mu_{\Psi_{P(\phi_t)}})$.
\end{proof}

For the case $t=1$ we already know by \cite{BRSS} that there is an acip, so the above proposition shows that the acip must have polynomial tails.  Hence the proof of Theorem~\ref{thm:poly} for (except for the proof that
$t\mapsto P(\phi_t)$ is analytic, which is postponed to the end of
Section~\ref{sec:tails}) essentially amounts to an application of
Proposition~\ref{prop:all_indu} (Case 4.) to the case $t\in (t_1,1)$, and is completed in a similar way to the proof of Theorem~\ref{thm:CE}.  The rate of decay of the tails follows from Proposition~\ref{prop:summable}.

The following lemma, which will be particularly useful in Section~\ref{sec:tails}, implies that we can fix an inducing scheme so that any measure with large free energy, for some $\phi_t$, must be compatible to this inducing scheme.

\begin{lemma}
For any point $x\in I$ there exists an inducing scheme $(X,F)$ as in
Lemma~\ref{lem:exten} with $x\in X$ and so that the following hold.
\begin{itemize}
\item In the case of, and with $t_1 < 1$ as in Theorem~\ref{thm:poly} (polynomial growth rate): for any $t_1<t_2<1$ there exists $\eps_0>0$ so that for all
$t\in (t_1,t_2)$, if
$h_\mu(f)+\int\psi_t~d\mu>P_+(\psi_t)-\eps_0$ then $\mu$ is compatible to $(X,F)$.

\item  In the case of, and with $t_1 < 1 < t_2$ as in Theorem~\ref{thm:CE}
(Collet-Eckmann):
there exists $\eps_0>0$ so that for all $t\in (t_1, t_2)$,  if
$h_\mu(f)+\int\psi_t~d\mu>P_+(\psi_t)-\eps_0$ then $\mu$ is compatible to $(X,F)$.
\end{itemize}

\label{lem:uniform scheme for nat}
\end{lemma}

\begin{proof}   By Lemma~\ref{lem:free energy bound}, there exist $\eps_0,\eps>0$ \st for any $t\in (t_1,t_2)$, $h_\mu(f)+\int\psi_t~d\mu>P_+(\psi_t)-\eps_0$ implies $h_\mu(f)>\eps$.  We can choose $\{\hat X^k\}_{0\le k\le N}$ as in Proposition~\ref{prop:unif scheme}: we need only select these sets so small that the corresponding inducing scheme is uniformly expanding, in order to satisfy (a) of that lemma, and so that $x\in \pi(\hat X^0)$.  Property (b)
of Proposition~\ref{prop:unif scheme}
follows for all $t\in (t_1,t_2)$ by the computations in the proof of Theorem~\ref{thm:CE} and in Proposition~\ref{prop:summable}.  The fact that for any $t\in (t_1,t_2)$, $\mu_t$ is compatible to our $(X^0, F_0)$ follows by Lemma~\ref{lem:eq sees all}.  Therefore, Proposition~\ref{prop:unif scheme} implies that the measures $\mu$ must be compatible to $(X^0,F_0)$. Finally take $(X,F)=(X^0,F_0)$.
\end{proof}

\section{Exponential Tails and Positive Discriminant}
\label{sec:tails}

In Theorems~\ref{thm:poly} and \ref{thm:CE} we see that
with the exception of non-Collet-Eckmann maps (i.e., satisfying
\eqref{eq:polynomial} but not \eqref{eq:CE}) with potential  $\phi
= -\log |Df|$, all the equilibrium states $\mu_\phi$ obtained are
compatible to an inducing scheme with exponential tail behaviour:
$\mu_\Psi(\{ x \in X : \tau(x) = n \}) \le C e^{-\alpha n}$ for
some $C, \alpha > 0$.

The literature gives many consequences; we mention a few:
\begin{itemize}
\item The system $(I,f,\mu_\phi)$ has exponential
decay of correlations and satisfies the Central Limit Theorem.
This follows directly from Young's results \cite{Y} relating the
decay of correlations to the tail behaviour of the Young tower.
\item The system $(I,f,\mu_\phi)$ satisfies the Almost Sure Invariance Principle (ASIP), see \cite{MN} or \cite{HKerg} for earlier ideas in this direction.
\item In \cite{Col}, Collet proves Gumbel's Law (which is related to
exponential return statistics) for the acip provided the Young
tower construction has exponential tail behaviour. It seem likely
that this result extends to the equilibrium states for $\phi_t =
-t \log|Df|$ and $t < 1$.
\end{itemize}
Another application of exponential tails pertains to analyticity
of the pressure function $t \mapsto P(\phi_t)$ and the absence of phase transitions (which would be expressed by lack of differentiability of the pressure function).
A key result here is phrased by Sarig \cite{Saphase}
in terms of directional derivatives
\[
\frac{d}{ds} P(\psi + s\upsilon)|_{s = 0}
\]
where $\psi$ and $\upsilon$ are suitable potentials.
To prove analyticity of $t \mapsto P(t\phi)$ near $t = 1$,
we take $\upsilon = \psi = \phi$.
Sarig obtains his results for Gurevich pressure. For appropriate potentials and inducing scheme, he first introduces the concept of discriminant $\dm$, which is positive if and only if the inducing scheme has exponential tails with respect to the equilibrium state of the induced potential.  Next it is shown that if the inducing scheme is a first return map, then positive discriminant implies analyticity of $s \mapsto P_G(\psi + s\upsilon)$ near $s = 0$.
In our case, the inducing scheme is a first return map on the Hofbauer tower, but also a Rokhlin-Kakutani tower can be constructed for which the first return map to the base is isomorphic to the inducing scheme.  Currently, in the context of smooth dynamical systems, these towers tend to be called a Young towers
\cite{Y}.
It is the better distortion properties than the Young tower on elements
of its natural partition $\Delta_{i,j}$, see below, that makes us prefer the Young tower over the Hofbauer tower in the section.

The resulting analyticity of the pressure function on the Young tower then needs to be related to the original system. We will do that using a transition from Gurevich pressure to the following type of pressure:
\[
P_+(\psi):=\sup\left\{h_\mu(f)+\int\psi~d\mu : \mu \in \M_+\hbox{ and } -\int\psi~d\mu<\infty \right\}
\]
for which we use a result by Fiebig et al. \cite{FFY}.

The set-up of the remainder of this section is as follows.
We first introduce the Young tower associated with the inducing scheme, and then discuss directional derivatives and discriminants.
This gives us the necessary terminology to state the main theorem
(Theorem~\ref{thm:tails and analytic}). Then we show how this can be applied to prove the remaining analyticity parts of Theorems~\ref{thm:poly} and \ref{thm:CE}.
Finally, we prove Theorem~\ref{thm:tails and analytic}.

Let $X \subset I$ and $(X,F,\tau)$ be an inducing scheme on $X$
where $F= f^\tau$. As usual we denote the set of domains of the
inducing scheme by $\{ X_i \}_{i \in \N}$.  The \emph{Young tower},
see \cite{Y}, is defined as the disjoint union
\[
\Delta = \bigsqcup_{i \in \N} \bigsqcup_{j = 0}^{\tau_i-1}
(X_i,j),
\]
with dynamics
\[
f_\Delta(x,j) = \left\{ \begin{array}{ll}
(x,j+1) & \mbox{ if } x \in X_i, j < \tau_i-1; \\
(F(x),0) & \mbox{ if } x \in X_i, j = \tau_i-1. \\
\end{array} \right.
\]
For $i\in \N$ and  $0\le j<\tau_i$, let $\Delta_{i,j} :=\{(x,j):
x\in X_i\}$ and $\Delta_l:=\bigcup_{i \in \N}\Delta_{i,l}$ is
called the {\em $l$-th floor}. Define the natural projection $\pi_\Delta:\Delta \to X$ by $\pi_\Delta(x,j) =
f^j(x)$, and $\pi_X:\Delta \to X$ by $\pi_X(x,j) = x$.
Note that $(\Delta, f_\Delta)$ is a Markov system, and the first return map of $f_\Delta$ to the \emph{base} $\Delta_0$ is isomorphic $(X,F,\tau)$.

Also, given $\psi:I\to \R$, let $\psi_\Delta:\Delta\to \R$ be defined by $\psi_\Delta(x,j) =\psi(f^j(x))$.  Then the induced potential of $\psi_\Delta$ to the first return map to $\Delta_0$ is exactly the same as the induced potential of $\psi$ to the inducing scheme $(X,F,\tau)$.

The differentiability of the pressure functional can be expressed
using directional derivatives $\left.\frac d{ds}P_G(\psi+s
\upsilon)\right|_{s=0}$.  We will use the method of \cite{Saphase},
but will require less stringent conditions on the potentials.  Let
$(W,f)$ be a topologically mixing dynamical system with the set
of  $n$-cylinders denoted by $\Q_n$.  For a potential
$\psi:W\to [-\infty,\infty]$ we can ask that $\psi$ satisfies
\begin{equation}
\sup_{C_n\in \P_n}\sup_{x,y\in
C_n}|\psi_n(x)-\psi_n(y)|=o(n).\label{eq:FFY} \end{equation}  As
shown in \cite{FFY}, this guarantees that $\psi$ satisfies
\eqref{eq:subadd} which means that the Gurevich pressure is well
defined and independent of the initial cylinder set $X_i$, where
$Z_n(\psi)=Z_n(\psi,X_i)$; also Theorem~\ref{thm:FFY} below is
satisfied.  Moreover, if the induced potential is weakly H\"older
continuous, then \eqref{eq:FFY} is a sufficient condition on the
original potential to allow us to use the results of \cite[Section
6]{Saphase}, see Theorem~\ref{thm:analytic}.

For an inducing scheme $(X,F, \tau)$, let $\psi_\Delta$ and
$\upsilon_\Delta$ be the lifted potentials to the Young tower.
Suppose that $\psi_\Delta:\Delta \to \R$ satisfies \eqref{eq:FFY}.
We define the set of \emph{directions} with respect to $\psi$ as the
set
\begin{align*}
Dir_F(\psi):=\Bigg\{\upsilon:\sup_{\mu\in \M_+}\left|\int\upsilon~d\mu\right|<& \, \infty, \upsilon_\Delta
\hbox{ satisfies \eqref{eq:FFY}}, \sum_{n=2}^\infty
V_n(\Upsilon)<\infty,\hbox{ and }\\
& \exists \eps>0 \hbox{ s.t. }
P_G(\psi_\Delta+s\upsilon_\Delta)<\infty \ \forall \ s \in
(-\eps,\eps) \Bigg\},
\end{align*}
where $\Upsilon$ is the induced potential of $\upsilon$. As in previous sections, let $\psi_S := \psi-S$ (and so $\Psi_S=\Psi-S \tau$).  Set $p_F^\ast[\psi]:=\inf\{S : P_G(\Psi_S)<\infty\}$.\footnote{Note that we use
the opposite sign for $p_F^\ast[\psi]$ to Sarig.}  If $p_F^\ast[\psi]>-\infty$, we define the \emph{$X$-discriminant}
of $\psi$ as
\[
\dm_F[\psi]:=\sup\{P_G(\Psi_S):S > p_F^\ast[\psi]\}\le \infty.
\]
Given a dynamical system $(X,F)$, we say that a potential $\Psi:X
\to \R$ is \emph{weakly H\"older continuous} if there exist
$C,\gamma>0$ \st $V_n(\Psi)\le C\gamma^n$ for all $n \ge 0$.

The main result of this section is as follows:

\begin{theorem} Let $f \in \H$ be an interval map with potential
$\phi:I\to (-\infty,\infty]$.  Suppose that $\phi$ satisfies \eqref{eq:FFY}
or is of the form $\phi=-t\log|Df|$. Take $\psi = \phi - P(\phi)$.  Then
$\dm_F[\psi]>0$ if and only if $(X,F,\mu_\Psi)$ has exponential tails.

Moreover, the inducing scheme can be chosen \st
given $\upsilon\in Dir_F(\psi)$ \st
$\psi_\Delta+\upsilon_\Delta$ is continuous and the induced potential
$\Upsilon$ is weakly H\"older continuous, there exists $\eps>0$ \st
$s \mapsto P_+(\psi+s \upsilon)$ is real analytic on $(-\eps,\eps)$.
\label{thm:tails and analytic}
\end{theorem}

As noted before, the appropriately shifted potential $\phi_t
= -t\log|Df|$, gives rise to an equilibrium state with exponential tail for $t$ in a neighbourhood of 1 if \eqref{eq:CE} holds, and for $t \in (t_1, 1)$ if  \eqref{eq:CE} fails but \eqref{eq:polynomial} holds.
Take $\upsilon = -\log|Df|$. Any induced system provided in Section~\ref{sec:poly} is extendible, so by the Koebe lemma the induced potential $\Upsilon$ has summable variations, and in fact is weakly H\"older.  Similarly $(-\log|Df|)_\Delta$ satisfies \eqref{eq:FFY}.
Also, since $P_G(\psi_\Delta+s\upsilon_\Delta)\le P_G(\Psi+s\Upsilon)$ which is clearly bounded for small $s$, we have the $P_G(\psi_\Delta+s\upsilon_\Delta)<\infty$ for small $s$. Therefore there is an inducing scheme
with $\upsilon \in Dir_F(\psi)$. Thus Theorem~\ref{thm:tails and analytic} can be applied to give the analyticity of $t \mapsto P(\phi_t)$ for $t\in (t_1,0)$, to complete the proofs of Theorems~\ref{thm:poly} and \ref{thm:CE}.

\begin{proof}
Suppose that $\dm_F[\psi]>0$. This is equivalent to the existence of $0>\eps_0>p_F^\ast[\psi]$ \st $P_G(\Psi_{\eps_0})<\infty$.  By the Gibbs property, for $\eps > \eps_0$ we have $\mu_{\Psi_\eps}(\{\tau=n\}) \asymp \sum_{\tau_i=n}e^{\Psi_i-n\eps}$.  Then $$\mu_{\Psi_\eps}(\{\tau=n\}) \asymp e^{-n(\eps-\eps_0)}\sum_{\tau_i=n}e^{\Psi_i-n\eps_0}.$$
Notice that $$\sum_{\tau_i=n}e^{\Psi_i-n\eps_0} \asymp
\mu_{\Psi_{\eps_0}}(\{\tau=n\})< \mu_{\Psi_{\eps_0}}(X)=1,$$ so
$\mu_{\Psi_\eps}(\{\tau=n\}) <Ce^{-n(\eps-\eps_0)}$.  Since
$\eps-\eps_0>0$, $(X,F,\mu_{\Psi_\eps})$ has exponential tails.

Conversely, suppose that $(X,F,\mu_{\Psi})$ has exponential tails
with exponent $\alpha>0$, that is
$$
\sum_{\tau_i=n}e^{\Psi_i} \asymp \mu_{\Psi}(\{\tau=n\})<
Ce^{-n\alpha}.
$$
Then, for all $-\alpha<\eps_0$, and for $Z_0$ defined on
page~\pageref{page:Z_0},
$$
P_G(\Psi_{\eps_0})\le CZ_0(\Psi_{\eps_0}) \le
C\sum_n\sum_{\tau_i=n}e^{\Psi_i-n\eps_0}< C\sum_n
e^{-n(\alpha+\eps_0)}<\infty.
$$
Therefore $p_F^\ast[\psi]\le -\alpha<0$ and so $\dm_F[\psi]>0$.

For the second part of the theorem, we use the following result
from \cite[Theorem 4]{Saphase}.
\begin{theorem}
Let $(W,f)$ be a topologically mixing dynamical system and
$\psi:W\to (-\infty,\infty]$ be a potential satisfying \eqref{eq:FFY}, \st $P_G(\psi)<\infty$ and for $X\in \P_n$,
$\dm_F[\psi]>0$ and $\Psi$ is weakly H\"older continuous. Then
for all $\upsilon\in Dir_F(\psi)$ \st $\Upsilon$ is weakly H\"older
continuous, there exists $\eps>0$ \st $s \mapsto
P_G(\psi+s \upsilon)$ is real analytic on $(-\eps,\eps)$.
\label{thm:analytic}
\end{theorem}

We can use this to show that $s \mapsto P_G(\psi+s\upsilon)$ is
analytic. However, to go from the Gurevich pressure to the usual
pressure, we need a Variational Principle.  Sarig's theory
provides various conditions on potentials which yield a
Variational Principle, but they are somewhat restrictive, and in
particular for our case, are not satisfied by the potential
$-t\log|Df|$.  One aim of \cite{FFY} is to weaken these
conditions. There, the following theorem is proved.

\begin{theorem}
If $(W,S)$ be a transitive Markov shift and $\psi:W\to \R$ is a continuous function satisfying \eqref{eq:FFY}, then $P_G(\psi)=P(\psi)$.\label{thm:FFY}
\end{theorem}

We now apply Theorem~\ref{thm:analytic} to the symbolic space induced by $(\Delta,f_\Delta)$.  In this space, the potential
$(-t\log|Df|-S')_\Delta$ satisfies \eqref{eq:FFY} and is continuous in the symbolic metric.
Theorem~\ref{thm:analytic} implies
that there is $\eps' > 0$ \st $s \mapsto P_G(\psi_\Delta+s \upsilon_\Delta)$
is analytic on $(-\eps',\eps')$.
Thus, by Theorem~\ref{thm:FFY}, $s\mapsto P(\psi_\Delta+s\upsilon_\Delta)$ is also
analytic on $(-\eps',\eps')$.

All $f_\Delta$-invariant probability measures $\nu$ have positive
Lyapunov exponents.  This is because the induced map $(X,F)$ (which is isomorphic to the first return map to $\Delta_0$) is uniformly expanding and the Ergodic  Theorem gives
\[
\lambda(\nu) :=  \int \log|Df_\Delta| \ d\nu
= \nu(\Delta_0) \int \log|DF_\Delta| \ d\nu \ge
\nu(\Delta_0) \inf_x \log |DF(x)| > 0.
\]
Therefore $P(\psi_\Delta+s\upsilon_\Delta) =
P_+(\psi_\Delta+s\upsilon_\Delta)$ for $s \in (-\eps',\eps')$.

Since the inducing scheme $(X,F)$ is obtained from both $(I,f)$ and
$(\Delta, f_\Delta)$ with the same inducing time $\tau = \tau_\Delta$,
Lemma~\ref{lem:Abra} implies that
\[
h_{\mu_\Delta}(f_\Delta) = \left( \int \tau d\mu_F \right)^{-1} h_{\mu_F}(F) =
h_\mu(f)
\]
and
\[
\mu_\Delta(\phi_\Delta) = \left( \int \tau d\mu_F \right)^{-1} \mu_F(\Phi) =
\mu(\phi),
\]
whenever $\mu_\Delta$ and $\mu_F$ are the induced measures of
$\mu$ to $(\Delta, f_\Delta)$ and $(X,F)$ respectively, and
$\phi$ is any potential.  Thus the free energy of $\mu$ and the lifted version $\mu_\Delta$ are the same.  This implies that  $s \mapsto P_G(\psi+s \upsilon)$
is analytic on $(-\eps',\eps')$ if the definition of pressure
involved only those measures which lift to $\Delta$.
Moreover, $P_+(\psi_\Delta+s\upsilon_\Delta)\le P_+(\psi+s\upsilon)$
for $s \in (-\eps', \eps')$.

It remains to prove that there exists $\eps>0$ so that for all $s\in
(-\eps, \eps)$, $P_+(\psi_\Delta+s\upsilon_\Delta) \ge
P_+(\psi+s\upsilon)$.  The issue is that in principle there might be
measures which have high free energy but do not lift to $\Delta$.
We show how Lemma~\ref{lem:uniform scheme for nat} implies that this is impossible, thus completing the theorem.  Since by
assumption
$\sup_{\mu\in \M_+}\left|\int\upsilon~d\mu\right|<\infty$,
$P_+(\psi+\eps\upsilon) \to P_+(\psi)=0$ as $\eps\to 0$.  Therefore
there exists $0<\eps<\eps'$ so that for any $s\in (-\eps,\eps)$,
we have $P_+(\psi+s\upsilon)>-\frac{\eps_0}2$.  Hence for all $s\in
(-\eps,\eps)$, if a measure $\mu$ has
$h_\mu(f)+\int\psi+s\upsilon~d\mu>P_+(\psi+s\upsilon)-\frac{\eps_0}2$ then Lemma~\ref{lem:uniform scheme for nat} implies
$\hat\mu(\hat X)>0$.  Hence
$P_+(\psi_\Delta+s\upsilon_\Delta) \ge P_+(\psi+s\upsilon)$.
Therefore $P_+(\psi_\Delta + s\upsilon_\Delta) = P_+(\psi+s\upsilon)$,
and the analyticity of
$s \mapsto P_+(\psi + s \upsilon)$ on $(-\eps,\eps)$ follows.
\end{proof}

It would be a further step to say that $t \mapsto \mu_{\phi_t}$ is analytic (where $\mu_{\phi_t}$ indicates the equilibrium state of $\phi_t$).  Using the weak topology we can ask whether $t \mapsto \int g \ d\mu_{\phi_t}$
is analytic for any fixed continuous function $g$.
We do have the following corollary:
\begin{corollary}
In the setting of Theorems~\ref{thm:poly} and \ref{thm:CE}, let $(X,F,\tau)$ be any inducing scheme as in Section~\ref{sec:induce}.
Fix $s \in (t_1,1)$ or $s$ in a small neighbourhood of 1, according to whether
\eqref{eq:polynomial} or \eqref{eq:CE} holds.  Take $\psi_t = \phi_t-P_+(\phi_s)$ for $\phi_t = -t\log|Df|$, and let $\Phi_t$ the induced potential. Then the function $t \mapsto \int_X \tau d\mu_{\Psi_t}$ is analytic for $t$ sufficiently close to $s$, where $\mu_{\Psi_t}$ denotes the equilibrium state of $\Psi_t$.
\end{corollary}

\begin{proof}
We know that $t \mapsto P_+(\psi_t)$ and $t \mapsto P(\Psi_t)$
are analytic.  By Lemma~\ref{lem:Abra}, $P(\Psi_t) = \left( \int \tau d\mu_{\Psi_t} \right) P_+(\phi_t)$, so analyticity of $t \mapsto \int \tau d\mu_{\Psi_t}$ follows.
\end{proof}

\section{Concerning the Hypotheses of Theorems~\ref{thm:poly} and
\ref{thm:CE}} \label{sec:hypo}

In this section, we argue that the hypotheses of
Theorems~\ref{thm:poly} and ~\ref{thm:CE} cannot
easily be relaxed. We also discuss some consequences of our
proofs.

{\bf The set \boldmath$\M_+$\unboldmath}:\label{nonM+}
The question how large the set
$\M_+$ is in comparison to $\M_{erg}$ is answered by Hofbauer and
Keller \cite{HKquad} in certain contexts. For unimodal maps, they
prove that any
measure $\mu \in \M_{erg} \setminus \M_+$ has entropy $0$ and
belongs to the convex hull of the set of weak accumulation points
of $\{ \frac1n \sum_{k=0}^{n-1} \delta_{f^k(c)} \}_{n \in \N}$,
where $\delta_{f^k(c)}$ indicates the Dirac measure at the $k$-th
image of the critical point. If we restrict to the potential
$\phi_t = - t  \log |Df|$ at $t = 1$, then the following examples
can be given:
\begin{itemize}
\item If $f$ has a neutral fixed point, then the Dirac measure at this
fixed point is an equilibrium state.
\item There is a quadratic map without equilibrium measure for $\phi_1$, see \cite{BrK}. In this case, the summability condition \eqref{eq:summable} fails.
\item  For maps such as the Fibonacci map (which satisfies \eqref{eq:polynomial} for $\ell=2$), there is only one
measure in $\M_{erg} \setminus \M_+$, namely the unique invariant
probability measure $\mu_{\omega(c)}$ supported on the critical
omega-limit set $\omega(c)$. This gives rise to a phase
transition for the pressure function $t \mapsto P(\phi_t)$ at $t =
1$. The quadratic Fibonacci map has two equilibrium states for
$\phi_1$: an absolutely continuous probability measure and
$\mu_{\omega(c)}$.

Moreover, there is a sequence of periodic points $p_n$
with Lyapunov exponents $\lambda(p_n) \searrow 0$ as $n \to \infty$, see
\cite{NSa}.
The equidistributions on $\orb(p_n)$ belong to $\M_+$, which shows that
$P_+(\phi_t) = 0$ for $t \ge 1$, but $\M_+$ contains no equilibrium states
if $t > 1$.
See \cite{BrK} for more information on the
phase transition.
\item It is also possible that $\M_{erg} \setminus \M_+$ contains
several equilibrium states, all supported on $\omega(c)$. In
\cite{Brminim} an example is given where $\omega(c)$ supports at least
two ergodic measures, while there is also an acip, as follows from
\cite[Theorem A (c)]{Brabsorb}.
\end{itemize}

{\bf Differentiability of the map \boldmath$f$\unboldmath:}\label{raith} A
$C^{1+\eps}$ assumption is necessary in order to use the result
that $\lambda(\mu) > 0$ implies liftability.  This result, proved
in \cite{Kellift}, relies on the property that $\mu$-typical
points have nondegenerate unstable manifolds, see \cite{Ledrap}.
If $f$ is only piecewise continuous, this property as well as
liftability no longer hold; this is illustrated by an example due to
Raith \cite{Raith}, see the left-hand graph in Figure~\ref{fig:raith}.
This is piecewise continuous map $f$ with slope $2$, having a
zero-dimensional set $H$ on which $f$ is semiconjugate to a circle rotation.  The unique $f$-invariant measure $\mu$ of $(H,f)$ has
$\lambda(\mu) = \log 2 > 0$, but cannot be lifted to the Hofbauer tower, described in Section~\ref{sec:induce}.   This follows since it can be shown that for each $x\in H$ and $\hat x\in \pi^{-1}(x)$, $\hat f^n(\hat x)$ belongs to a domain $D_n\in \D$ and $\lim_{n\to\infty}|D_n|\to 0$.  As shown in the graph on the right of Figure~\ref{fig:raith}, is easy to adjust this example into a
continuous map with slope $\pm 2$, but this map is not differentiable at the turning points. Another part where $C^2$ differentiability is used is Ma\~n\'e's Theorem in the proof of Proposition~\ref{prop:summable}.

\begin{figure}[ht]
\begin{center}
\begin{minipage}{120mm}
\unitlength=7mm
\begin{picture}(22,10)(1.5,0)
\let\ts\textstyle

\put(-0.4,2.1){\scriptsize $\frac12-2\alpha$}
\put(4,0.5){\scriptsize $\frac14 + \alpha$}
\put(9.7,3.25){\scriptsize $1-4\alpha$} \put(17,0.5){\scriptsize
$\frac12 + 2\alpha$}

\thinlines \put(1,1){\line(1,0){8}} \put(1,5){\line(1,0){8}}
\put(1,9){\line(1,0){8}} \put(9,1){\line(0,1){8}}
\put(1,1){\line(0,1){8}} \put(5,1){\line(0,1){8}}

\thicklines \put(5,5){\line(1,2){2}} \put(7,5){\line(1,2){2}}

\put(1,2.2){\line(1,2){1.4}}\put(1,2.23){\line(1,2){1.37}}
\put(1,2.17){\line(1,2){1.4}}
 \put(4.4,1){\line(1,2){0.6}}\put(4.4,1.03){\line(1,2){0.58}}
\put(4.4,0.97){\line(1,2){0.6}} \put(3.4,5){\line(1,2){1}}
\put(2.4,7){\line(1,2){1}} \put(7,5){\line(1,2){2}}

\thinlines \put(11,1){\line(1,0){8}}
\put(11,9){\line(1,0){8}} \put(19,1){\line(0,1){8}}
\put(11,1){\line(0,1){8}}

\thicklines

\put(11,3.4){\line(1,2){2.8}} \put(11,3.43){\line(1,2){2.78}}
\put(11,3.37){\line(1,2){2.82}} \put(17.8,1){\line(1,2){1.2}}
\put(17.8,1.03){\line(1,2){1.17}} \put(17.8,0.97){\line(1,2){1.2}}
\put(13.8,9){\line(1,-2){4}}
\end{picture}
\caption{{\bf Left}: Raith's example. For specific choices of
$\alpha$, the points whose orbits stay in the domains of branches 1
and 4 (bold lines) for ever form a zero-dimensional Cantor set
$H$ on which $f$ is semi-conjugate to a circle rotation.
\newline {\bf Right}: Rescaling the left bottom square and inserting a new branch gives a continuous example.  Again the set of points whose orbits stay in the domains branches 1 and 3 (bold lines) for ever form a zero-dimensional Cantor set $H$ on which $f$ is semi-conjugate to a circle rotation.}\label{fig:raith}
\end{minipage}
\vskip-20pt
\end{center}
\end{figure}

{\bf Measures with \boldmath$\mbox{supp}(\mu) \subset \mbox{orb}(\mbox{Crit})$\unboldmath:}
\label{masm}
Makarov and Smirnov \cite{MSm1,MSm2} discuss specific polynomials $f$ on
the complex plane for which there is a phase transition for the
potential $\phi_t = -t \log|Df|$ at some $t < 0$, and consequently
these example would contradict our main theorem. The reason for
this is that the Julia set $J(f)$ has `very exposed' fixed points
on which the Dirac measures can become equilibrium states for $t$
sufficiently small. In the interval setting this applies to the
Chebyshev polynomials $f:[0,1] \to [0,1]$ of any degree $d \ge 2$.
The set $\{ 0, 1\}$ consists of the critically accessible points; each
critical point is prefixed, and either (a) $0 = f(0) = f(1) =
f^2(\Crit)$; or (b) $0 = f(0)$, $f(1) = 1$ and $0$ and $1$ are both
critical values of critical points. The critical
accessibility creates an obstruction in our strategy of finding an
induced scheme in Section~\ref{sec:induce}.
Further results on phase transitions for $t > 1$ are given in
\cite{MSm3}.

{\bf The Gibbs property:}\label{gibbs} Although the equilibrium
states obtained in $\M_+$ (i.e., for the original system) are
positive on open sets, we cannot expect them to be Gibbs. First,
if $\phi = -\log|Df|$, then $\phi$ is unbounded near critical
points, so it is impossible to have $e^{\phi_n(x) - nP(\phi)} \le
K \mu(\c_n[x])$ uniformly in $x$. But also if the number $K$ is
allowed to depend on $x$, measures cannot always satisfy this
weaker form of the Gibbs property. For example, if $f(x) =
ax(1-x)$ has an acip $\mu$, and the potential is $\phi = -\log
|Df|$, then the pressure $P(\phi) = 0$ and it is well known that
$\frac{d\mu}{dx} > \rho_0 > 0$ on a neighbourhood of $c$. Suppose by
contradiction that for each $x \notin \cup_{n \in \Z} f^n(c)$,
there exists $K = K(x)$ \st
\[
\frac{1}{K} \le \frac{\mu(\c_n[x])}{e^{\phi_n(x)}} \le K \mbox{
for each }  n \ge 0.
\]
Now $\mu$-a.e. $x$ has an orbit accumulating on $c$, so almost
surely there exists $n$ such $|f^n(x) - c| < \frac1{4 K^2}$. But
then
\begin{align*}
\mu(\c_{n+1}[x]) &\ge \frac1{K} e^{\phi_{n+1}(x)} =  \frac1{K}
e^{\phi_{n}(x)} \frac{1}{|Df^n(x)|} \ge \frac{1}{K^2} \mu(\c_n[x])
\frac{4 K^2}{2} \ge 2\mu(\c_n[x]),
\end{align*}
which contradicts that $\c_{n+1}[x] \subset \c_n[x]$. Thus $\mu$
cannot be a Gibbs measure.

In some cases, a weak Gibbs property can be proved. For example,
it was shown in \cite{BrV} that for unimodal maps with critical
order $\ell$ satisfying a summability condition, and every $\eps > 0$,
there exists $K = K(x)$ for Lebesgue a.e. $x$ \st
\[
\frac{1}{K n^{3(\ell+1)} } \le
\frac{\mu_\phi(\c_n[x])}{e^{\phi_n(x)} } \le K
n^{2(1+\eps)}.
\]

\section*{Appendix}

In this appendix we give the two remaining proofs.
The first is a lemma on the structure of the Hofbauer tower.

\begin{proof}[Proof of Lemma~\ref{lem:trans_subgraph}]
We start with case (a), so $\Omega$ is a finite union of intervals.
Let $x \in \Omega$ be any point with a dense orbit
in $\Omega$. Suppose that
$(\E, \to)$ is a maximal primitive subgraph that is not closed,
then for any $\hat x \in \pi^{-1}(x) \cap D_0$ for some $D_0 \in \E$,
$\orb(\hat x)$ leaves $\E$, i.e. $\hat f^k(\hat x)\notin \E$ for $k$ sufficiently large.  Indeed, since $\E$ is not closed, there is $D \in \E$ and $D' \notin \E$ \st $D \to D'$. There is an $n$-path $D_0 \to \dots \to D$ for arbitrarily large $n$, corresponding to sets $\hat\c_n \in \hat \P_n$. Each
$\hat \c_n$ has an $n+1$-subcylinder $\hat \c_{n+1}$ corresponding
to the $n+1$-path $D_0 \to \dots \to D \to D'$. For $n$
sufficiently large, $\hat \c_{n+1}$ is compactly contained in $D$.
Since $\orb(x)$ is dense in $\Omega$, there is $m$ \st $f^m(x)
\in \pi(\hat \c_{n+1})$. Therefore $\hat f^m(\hat x) \in \pi^{-1}
\circ \pi(\hat \c_{n+1})$ and $\hat f^{m+n+1}(\hat x) \in D''$ for
some domain \st $\pi(D'') \subset \pi(D')$. Regardless of
whether $D'' = D'$ or not, there is no path from $D''$ back into
$\E$, because if there was, there would be a path from $D'$ back
into $\E$, contradicting maximality of $\E$.

Consequently, $\orb(\hat x)$ will leave every maximal primitive subgraph
that is not closed.  If there is a closed primitive subgraph $(\E,
\to)$, then it is unique, $\hat f^k(\hat x)\in \E$ for all
sufficiently large $k$ and necessarily $\pi(\cup_{D \in \E} D) \supset
\Omega$. Let us also show that there is $\hat y$ with a dense orbit in $\E$.
Fix $D_0 \in \E$ and let $U_n$ be a countable base
of $\sqcup_{D \in \E} D$. Each $U_n$ intersects some $D$
and $U_n$ contains an $r_n$-cylinder $\hat\c_{r_n}\in \hat\P_n$ which itself is contained in $D$.
Since $\E$ is primitive, there is a path $D_0 \to \dots \to D$ of length $l_n$ and another path $D \to \dots \to D_0$ of length $l'_n \ge r_n$ \st if $\hat z \in D$ takes this
path, then $\hat z \in \hat\c_{r_n}$.   Let $p_n:=l_n+l_n'$.  Because $(\E, \to)$ is a Markov graph, for each $n\ge 1$ we have a cylinder $\hat\c_{p_n}\subset D_0$ \st $\hat f^{l_n}(\hat\c_{p_n})\subset \hat\c_{r_n}\subset U_n$ and $\hat f^{p_n}(\hat\c_{p_n})=D_0$.

Let $q_0=0$ and $q_n:=\sum_{k=1}^{n}p_k$. Let $\hat\c_{q_1}=\hat\c_{p_1}$.  By the Markov structure, we can pull back inductively to obtain a nested sequence of cylinder sets $\hat\c_{q_n}\subset \cdots \subset \hat\c_{q_1} \subset D_0$ with $\hat f^{q_{n}+l_{n+1}}(\hat \c_{q_{n+1}}) \subset U_{n+1}$ and  $\hat f^{q_{n}}(\hat \c_{q_{n+1}}) = \hat\c_{p_{n+1}}$ for all $n\ge 0$.  The point $\hat y\in \bigcap_n\hat\c_{q_n}$ has a dense orbit in $\E$.  In this case the lemma is proved.

Alternatively, suppose that no closed primitive subgraph exists.
Abbreviate $\hat \Omega_R := \pi^{-1}(\Omega) \cap \hat I_R$.
If $\# (\orb(\hat x) \cap \hat \Omega_R) =
\infty$ for some $R$, then $\# (\orb(\hat x) \cap D) = \infty$ for
some $D \subset \hat \Omega_R$, and $\hat f^k(\hat x)$ is in the non-empty
maximal primitive subgraph containing $D$, for all sufficiently
large $k$. The above argument shows that this subgraph is closed
as well, so we would be in the previous case after all.

Therefore $\orb(\hat x)$ has a finite intersection with every compact
subset of $\hat I$.
We will show that this contradicts $\orb(x)$ being dense in $I$,
by showing that $\orb(x)$ cannot accumulate on an orientation
reversing fixed point $p$, leaving the (very similar) argument
where $p$ is orientation preserving and/or where $p$ has a higher
period to the reader.

Assume (for the moment) that all critical points are turning
points (and not inflection points).
\begin{figure}[ht]
\begin{center}
\begin{minipage}{120mm}
\unitlength=7mm
\begin{picture}(22,5)(0,0)
\let\ts\textstyle

\thicklines \put(4,3){\line(1,0){6}}

\put(7,2){\line(1,0){0.5}}
\put(7,1.97){\line(1,0){0.5}}\put(8.1,1.9){\scriptsize $\pi(D_k)$}
\put(8,1){\line(-1,0){1.5}}
\put(8,0.98){\line(-1,0){1.5}}\put(8.1,0.9){\scriptsize
$\pi(D'')$} \put(4,1){\line(1,0){2}}
\put(4,0.98){\line(1,0){2}}\put(3.2,0.4){\scriptsize $\pi(D) =
f^n(\pi(D''))$}

\put(12,3){\line(1,0){3}}\put(15.2, 2.9){\scriptsize
$\pi(D_{k-l})$}

\put(1,3){\line(1,0){2}}\put(1.1, 1.6){\scriptsize $D^* \subset \hat \Omega_R$}
\put(1.3, 2.2){\scriptsize $\pi(D^*)$}
\put(4.9,2.97){\line(1,0){0.4}} \put(4.9,2.94){\line(1,0){0.4}}
\put(12.5,3.03){\line(1,0){0.5}} \put(12.5,3.06){\line(1,0){0.5}}
\put(14,3.03){\line(1,0){0.5}} \put(14,3.06){\line(1,0){0.5}}

\thinlines \put(4.5,3){\line(0,1){0.1}} \put(4.4, 3.4){\scriptsize $\zeta_0$}
\put(5.5,3){\line(0,1){0.1}} \put(5.4, 3.4){\scriptsize $\zeta_2$}
\put(4.9,3){\line(0,-1){0.1}}
\put(5.3,3){\line(0,-1){0.1}}
\put(4.8, 2.5){\scriptsize $\c^*_R$}
\put(6.8,3){\line(0,1){0.1}} \put(6.6, 3.4){\scriptsize $\zeta_n$}
\put(7.8,3){\line(0,1){0.1}} \put(7.6, 3.4){\scriptsize $\zeta_{n+2}$}

\put(9,3){\line(0,1){0.1}} \put(8.9, 3.4){\scriptsize $p$}

\put(12.5,3){\line(0,1){0.1}} \put(12.6, 3.4){\scriptsize $\c_l$}
\put(13,3){\line(0,1){0.1}} \put(14.1, 3.4){\scriptsize $\c''_l$}

\put(14,3){\line(0,1){0.1}} \put(14.5,3){\line(0,1){0.1}}

\put(5,3.2){\line(0,1){1}}\put(5,4.2){\line(-1,0){3}}
\put(2,4.2){\vector(0,-1){0.7}} \put(3.1, 4.4){\scriptsize $f^R$}

\put(12.75,2.8){\line(0,-1){0.3}}\put(12.75,2.5){\line(-1,0){5.5}}
\put(7.25,2.5){\vector(0,-1){0.2}}

\put(14.25,2.8){\line(0,-1){2.3}}\put(14.25,0.5){\line(-1,0){7}}
\put(7.25,0.5){\vector(0,1){0.2}}

\put(6.7,1.7){\line(0,-1){0.5}}\put(6.7,1.7){\line(-1,0){1.7}}
\put(5, 1.7){\vector(0,-1){0.4}} \put(5.6, 1.9){\scriptsize $f^n$}

\put(11, 0.7){\scriptsize $f^l$} \put(11, 2){\scriptsize $f^l$}

\end{picture}
\caption{The $\pi$-images of domains $D = D_k$ and $D'$, their
positions with respect to $\zeta_n$ and a sketch how this leads to
a path from $D_{k-l}$ back into $\hat \Omega_R$. } \label{fig:betaD}
\end{minipage}
\vskip-20pt
\end{center}
\end{figure}
Call $\zeta$ a \emph{precritical point of order} $k$ if
$f^k(\zeta) \in \Crit$ and $f^i(\zeta) \notin \Crit$ for $i < k$.
Let $p$ be an orientation reversing fixed point and $\zeta_0$ be a
precritical point \st $(\zeta_0, p)$ contains no precritical
point of lower order. Then there is a point $\zeta_1 \in
f^{-1}(\zeta_0)$ at the other side of $p$ with no precritical
point of lower order in $(p, \zeta_1)$. Continue iterating
backwards to find a sequence $\zeta_0 < \zeta_2 < \zeta_4 < \dots
< p < \dots < \zeta_5 < \zeta_3 < \zeta_1$, \st
$(\zeta_n,p)$ (or $(p, \zeta_{n+1})$) contains no precritical
point of lower order. Let $R$ be \st $(\zeta_0, \zeta_2)$
compactly contains an $R$-cylinder $\c^*_R$. It follows that if
$D$ is a domain \st $\pi(D) \supset (\zeta_0, \zeta_2)$,
then there is an $R$-path from $D$ leading to $D^* \subset \hat \Omega_R$, see
Figure~\ref{fig:betaD}. To continue the argument, we need the
following claim which is proved at the end of this proof.

\begin{claim*} Take $\eps := \min\{ |c-c'| : c \neq c' \in \Crit\}$,
fix $l \ge 0$ and let $J$ be any interval \st  $|f^i(J)| < \eps$
for all $i \le l$.
Then for any pair of $l$-cylinders $\c_l, \c'_l \subset J$, there
is an $l$-cylinder $\c''_l$ in the convex hull of $\c_l$ and
$\c'_l$ \st the images $f^l(\c_l), f^l(\c'_l) \subset
f^l(\c''_l)$.
\end{claim*}

Let $D_k$ be the domain containing $\hat f^k(\hat x)$. Recall that
for every maximal primitive non-closed subgraph $\E$, $D_k\in \E$
for at most finitely many $k$. So let $k_0$ be \st $D_{k_0}$
does not belong to any maximal primitive subgraph that intersects
$\hat \Omega_R$. It follows that for each $k \ge
k_0$, there is no path from $D_k$ leading back into $\hat \Omega_R$.
Furthermore, if $\limsup_k |D_k| \ge \eps$, where $\eps$ is as in
the claim, then for arbitrarily large $k$, there are paths $D_k$
leading back into $\hat \Omega_R$. Therefore we can take $k_0$ so large
that $|D_k| < \eps$ for all $k \ge k_0$.

Assume by contradiction that $p \in \overline{\orb(x)}$. Then
there are arbitrarily large $n$ \st if $k = k(n)$ is the
first integer \st $f^k(x) \in (\zeta_n, \zeta_{n+1})$, then
$k > k_0$. Now if $\pi(D_k) \supset (\zeta_n, \zeta_{n+2})$, then
there is an $n$-path from $D_k \to \dots \to D$ where $\pi(D)
\supset (\zeta_0, \zeta_2)$, and hence an $n+R$-path leading back
into $\hat \Omega_R$ (as in Figure~\ref{fig:betaD}). This contradicts
the definition of $k_0$.

Otherwise, i.e., if $\pi(D_k) \not\supset (\zeta_n, \zeta_{n+2})$,
then the claim implies that there exist $l$ and $l$-cylinders
$\c_l, \c''_l \subset \pi(D_{k-l})$ \st $f^l(\c_l) =
\pi(D_k)$ while $D''$ is \st $\pi(D'') = f^l(\c''_l) \supset
\pi(D_k)$ and $\pi(D'') \supset (\zeta_n, \zeta_{n+2})$, see
Figure~\ref{fig:betaD}. Take $l$ minimal with this property. As
before, this gives an $l+n+R$-path leading from $D_{k-l}$ to
$\hat \Omega_R$. If $k-l > k_0$, then we have a contradiction again
with the choice of $k_0$. However, we can repeat the argument for
infinitely many $n$, and hence infinitely many $k$. If $D_{k-l}$
has been used for one value of $k$, then at least one domain in
$\hat f(D_{k-l})$ is the starting domain of a path leading into
$\hat \Omega_R$. Minimality of $l$ implies that the same $D_{k-l}$ no
longer serves for the next value of $k$. This proves that for $n$
sufficiently large, $k-l > k_0$, and this contradicts the choice
of $k_0$, proving the lemma.

Finally, if there are critical inflection points, then we can
repeat the argument with a branch partition and Hofbauer tower
that disregards the inflection points. Indeed, the above arguments
made use only of the topological structure of $f$, so whether
$f|_{\c_1}$ is diffeomorphic or only homeomorphic on $\c_1 \in
\P_1$ makes no difference.

\begin{proof}[Proof of the Claim.] Let $J$ be an interval \st $|J| < \eps$ . We argue by induction. For $l = 1$, the claim is true, since $J$ can contain at most one $1$-cylinder. Suppose now the
claim holds for all integers $< l$ and $|f^i(J)| < \eps$ for all
$i \le l-1$. Let $\c_l,\c'_l \subset J$ be $l$-cylinders,
contained in $l-1$-cylinders
 $\c_{l-1},\c'_{l-1}$. By induction, we can find an
$l-1$-cylinder $\c''_{l-1}$ in the convex hull $[\c_{l-1},
\c'_{l-1}]$ \st $f^{l-1}(\c_{l-1}), f^{l-1}(\c'_{l-1})
\subset f^{l-1}(\c''_{l-1})$. If $\Crit \cap  f^{l-1}(\c''_{l-1})
= \emptyset$ then $\c''_{l-1}$ is also an $l$-cylinder and
$f^{l}(\c_l), f^l(\c'_l)  \subset f^l(\c''_{l-1})$,
proving the induction hypothesis for $l$. Otherwise, by definition
of $\eps$, $f^{l-1}(\c''_{l-1})$ contains a single critical point,
and the $f^l$-image of one $l$-subcylinder of $\c''_{l-1}$
contains the $f^l$-image of the other. It is easy to see that
this $l$-subcylinder satisfies the claim.
\end{proof}
This completes the proof of the claim and hence of
part (a) of Lemma~\ref{lem:trans_subgraph}.
Part (b) deals with renormalisable maps, so assume
that $J \neq I$ is a $p$-periodic interval
which is minimal in the sense that no proper subinterval of $J$ has period
$p$. We claim that $J$ is associated with an absorbing
subgraph $(\E_{\mbox{\tiny absorb}}, \to)$ of $(\D, \to)$.
Indeed, by minimality of $J$, $f^p:J \to J$ is onto,
and for any $x \in \orb(J)$ and $n \ge 0$, there is $x_n \in \orb(J)$
such that $f^n(x_n) = x$.
Let $\hat J = \cap_k \hat f^k(\pi^{-1}(\orb(J)))$. This set has the
following properties:
\begin{itemize}
\item $\hat J \neq \emptyset$: Since $J$ contains an (interior) $p$-periodic
point, it lifts to a $p$-periodic point in $\hat J$.
\item If $\hat x \in \hat J$ and $D \in \D$ is the domain containing $\hat x$,
then $D \subset \hat J$. This follows from the Markov property.
Let $x = \pi(\hat x)$, take $x_n \in \orb(J)$
as above and
$\hat x_n \in \pi^{-1}(\orb(J))$ such that $\hat f^{n}(\hat x_n) = \hat x$.
For $\hat y \in D$ arbitrary, we can find $\hat y_n \in \hat Z_n[\hat x_n]$
such that $\hat f^n(\hat y_n) = \hat y$. Since this holds for all
$n \in \N$, $\hat y \in \hat J$.
\item $\hat J$ is $\hat f$-invariant. This is immediate
 from  the $f$-invariance of $\orb(J)$ and the definition of $\hat J$.
\end{itemize}
Take $\E_{\mbox{\tiny absorb}} :=
\{ D \in \D : D \cap \hat J \neq \emptyset\}$.
Then the $\hat f$-invariance of $\hat J$ implies that
$(\E_{\mbox{\tiny absorb}}, \to)$ is indeed absorbing.
Now apply part (a) to the subgraph
$(\D \setminus \E_{\mbox{\tiny absorb}}, \to)$ to find the
required (non-closed) primitive subgraph.
\end{proof}

The next proof shows that measures of positive entropy must lift to cover a large portion of the Hofbauer tower.

\begin{proof}[Proof of Lemma~\ref{lem:compat entro}]
Liftability of $\mu$ was shown by Keller \cite{Kellift}, so it
remains to show that $\hat\mu(\hat I_R) > \eta$ uniformly over all
measures with $h_\mu(f) \ge \eps$.

Fix $R \in \N$ and $\delta > 0$ \st $(\delta + \frac2R)
\log(1+\#\Crit) < \eps/2$. Let $\P_n^u$ be the collection of
$n$-cylinders \st $\frac1n \# \{ k < n : \hat f^k \circ i(\c_n)
\subset \hat I_R\} < \delta$, where as before $i^{-1} =
\pi|_{D_0}$, and let $\P_n^l$ be the remaining $n$-cylinders.

If $\hat \mu(\hat I_R)$  is small, then $\mu(\cup_{\c_n \in
\P_n^l} \c_n)$ is small as well. Hence, if the lemma was false,
then for any $\eta > 0$ we could find a measure $\mu$ with $h_\mu(f)
\ge \eps$ and $\mu(\cup_{\c_n \in \P_n^l} \c_n) <
\frac{\eps}{2\log(1+\# {\rm Crit})}$. So assume by contradiction
that there is such a measure $\mu$.

If  $D \in \D$ is any domain outside $\hat I_R$, then only the two
outermost cylinder sets in $\P_R \cap D$ can map under $\hat f^R$
to domains of level $> R$. The $\hat f^R$-images of the other
cylinder sets $J'$ have both endpoints of level $\le R$, so they
have $\lev(\hat f^R(J')) \le R$. Repeating this argument for $\hat
f^R(J')$ of those outermost cylinder sets, we can derive that for
infinitely many $n$:
\[
\lambda_u^n := \# \P_n^u \le (1+\# \Crit)^{\delta n}(1+\#
\Crit)^{(1-\delta) 2n/R} \mbox{ and } \lambda_l^n :=
\# \P_n^l \le (1+\# \Crit)^{n},
\]
so $\log \lambda_u \le (\delta + \frac{2}{R}) \log (1+\# \Crit) <
\eps/2$ and $\log \lambda_l \le \log(1+\# \Crit)$. For any finite
set of nonnegative numbers $a_k$ \st $\sum_k a_k = a \le 1$,
Jensen's inequality gives $-\sum_k a_k \log a_k \le a \log \# \{
a_k \}$. Since the branch partition $\P$ is assumed to generate
the Borel $\sigma$-algebra, the entropy of $\mu$ can be computed
as
\begin{align*}
h_\mu(f) &= \inf_n -\frac{1}n \sum_{\c_n \in \P_n} \mu(\c_n)
\log \mu(\c_n) \\
&= \inf_n -\frac{1}n \left( \sum_{\c_n \in \P_n^l} \mu(\c_n) \log
\mu(\c_n) +
\sum_{\c_n \in \P_n^u} \mu(\c_n) \log \mu(\c_n)  \right) \\
&\le  \inf_n \frac1n \left( \frac{\eps}{2(1+\#\Crit)} \log
\lambda_l^n + \log \lambda_u^n\right) < \eps.
\end{align*}
This contradiction establishes the required $\eta > 0$.

Now to prove the second statement, for each $D \subset \hat I_R$,
we can find $\kappa_D > 0$ \st if $\hat x \in D$ and $d(\hat
x, \partial D) < \kappa_D$, then $\hat f^k(\hat x) \notin \hat
I_R$ for $R < k \le 3R/\eta$. Obviously the set $\hat E := \cup_{D
\subset \hat I_R} \{ \hat x \in D : d(\hat x, D) > \kappa_D\}$ is
compactly contained in $\hat I_R$. If $\hat x$ is a typical point
for $\hat \mu$, then the relative time of $\orb(\hat x)$ spent
outside $\hat I_R$ is at least $\hat \mu(\hat I_R \setminus \hat
E) (\frac{3}{\eta}-1) \le 1$, so  $\hat \mu(\hat I_R \setminus
\hat E) < \eta/2$, whence $\hat\mu(\hat E) > \eta/2$.
\end{proof}

\medskip
\noindent
Department of Mathematics\\
University of Surrey\\
Guildford, Surrey, GU2 7XH\\
UK\\
\texttt{h.bruin@surrey.ac.uk}\\
\texttt{http://www.maths.surrey.ac.uk/}

\medskip
\noindent
Department of Mathematics\\
University of Surrey\\
Guildford, Surrey, GU2 7XH\\
UK\footnote{
{\bf Current address:}\\
Departamento de Matem\'atica Pura\\
Rua do Campo Alegre, 687\\
4169-007 Porto\\
Portugal\\
}\\
\texttt{mtodd@fc.up.pt}\\
\texttt{http://www.fc.up.pt/pessoas/mtodd/}

\end{document}